\documentclass[review,a4paper]{amsart}
\usepackage[T1]{fontenc}
\usepackage[left=3cm,right=3cm,top=3cm,bottom=3cm]{geometry}
\usepackage{mathtools,bbm,amsthm,amsmath,amsfonts}
\usepackage{amssymb,color,epic,multirow,comment}
\usepackage{algorithmicx,algorithm,pifont}
\usepackage{graphicx}
\usepackage{subcaption}
\usepackage[export]{adjustbox}
\usepackage{tikz,pgfplots}
\usepackage{pgfplotstable}
\usepackage{orcidlink}
\usepackage{mathrsfs}
\usepackage{cite}
\usepgfplotslibrary{groupplots}
\pgfplotsset{compat=newest}
\usetikzlibrary{arrows,decorations,backgrounds,positioning}
\usepgfplotslibrary{colorbrewer}
\mathtoolsset{showonlyrefs}
\usepackage{lineno}
\usepackage{setspace}
\theoremstyle{plain}

\newtheorem{theorem}{Theorem}[section]
\newtheorem{lemma}[theorem]{Lemma}
\newtheorem{proposition}[theorem]{Proposition}
\newtheorem{corollary}[theorem]{Corollary}
\newtheorem{definition}[theorem]{Definition}
\newtheorem{remark}[theorem]{Remark}

\def\Letters{A,B,C,D,E,F,G,H,I,J,K,L,M,N,O,P,Q,R,S,T,U,V,W,X,Y,Z}
\makeatletter
\@for \@l:=\Letters \do{%
  \expandafter\edef\csname\@l bb\endcsname{%
  \noexpand\ensuremath{\noexpand\mathbb{\@l}}}%
  \expandafter\edef\csname\@l bf\endcsname{{\noexpand\bf \@l}}%
  \expandafter\edef\csname\@l cal\endcsname{%
  \noexpand\ensuremath{\noexpand\mathcal{\@l}}}%
  \expandafter\edef\csname\@l eu\endcsname{%
  \noexpand\ensuremath{\noexpand\EuScript{\@l}}}%
  \expandafter\edef\csname\@l frak\endcsname{%
  \noexpand\ensuremath{\noexpand\mathcal{\@l}}}%
  \expandafter\edef\csname\@l rm\endcsname{{\noexpand\rm \@l}}%
  \expandafter\edef\csname\@l scr\endcsname{%
  \noexpand\ensuremath{\noexpand\mathscr{\@l}}}%
}
\newcommand{\bs}[1]{{\boldsymbol#1}}
\renewcommand{\d}{\operatorname{d\!}}

\newcommand{\gfrac}[2]{\genfrac{}{}{0pt}{}{#1}{#2}}

\newcommand{\isdef}{\mathrel{\mathrel{\mathop:}=}}

\newcommand{\spn}{{\operatorname{span}}}

\definecolor{navy}{RGB}{102,153,255}
\definecolor{tuerkis}{RGB}{51,153,204}
\algrenewcommand\alglinenumber[1]{\ding{\taumexpr191 + #1}}

\title[Samplet limits and multiwavelets]
{Samplet limits and multiwavelets}

\author[G. Giacchi]{Gianluca Giacchi \orcidlink{0000-0002-6809-1311}}
\author[M. Multerer]{Michael Multerer \orcidlink{0000-0003-0170-0239}}
\author[J. Quizi]{Jacopo Quizi \orcidlink{0009-0001-9199-2812}}

\address{IDSIA USI-SUPSI,
Universit{\`a} della Svizzera italiana,
Via la Santa 1, 6962 Lugano, Switzerland}
\email{
  \href{mailto:gianluca.giacchi@usi.ch}{gianluca.giacchi@usi.ch},
  \href{mailto:michael.multerer@usi.ch}{michael.multerer@usi.ch},
  \href{mailto:jacopo.quizi@usi.ch}{jacopo.quizi@usi.ch}
}

\begin{document}
\begin{abstract}
Samplets are data adapted multiresolution analyses of localized discrete signed
measures. They can be constructed on scattered data sites in arbitrary dimension
such that they exhibit vanishing moments with respect to any prescribed set
of primitives. We consider the samplet construction in a probabilistic framework
and show that, if choosing polynomials as primitives, the resulting samplet
basis converges to signed measures with broken polynomial densities in the
infinite data limit. These densities amount to multiwavelets with respect to a
hierarchical partition of the region containing the data sites. As a byproduct,
we therefore obtain a construction of general multiwavelets that allows
for a flexible prescription of vanishing moments going beyond
tensor product constructions. For congruent partitions we particularly
recover classical multiwavelets with scale- and partition- independent filter 
coefficients. The theoretical findings are complemented by numerical experiments 
that illustrate the convergence results in case of random as well as
low-discrepancy data sites.
\end{abstract}

\maketitle
\section{Introduction}\label{sec:intro}
Multiresolution analyses are a central tool in approximation, compression and 
fast computation. In many settings, one is interested in an orthonormal,
or at least stable, basis that is spatially localized and ordered by scale,
such that smooth components of a given signal are represented on coarse scales, 
while fine scales capture details. In the univariate setting and in
tensor-product regimes, wavelet bases are the canonical choice for this purpose. 
In many applications, however, data are scattered, high-dimensional
and supported on irregular domains. In such settings, tensor-product constructions
are often inadequate, as they suffer from the curse of dimensionality 
and do not adapt well to non-uniform sampling or irregular geometry. 
This motivates the search for data-adapted multiresolution constructions.

Among classical multiresolution techniques, compactly supported orthogonal
wavelets and polynomial multiwavelets are particularly relevant to the present
work. These constructions yield sparse approximations and operator compression
and play an important role in numerical analysis~\cite{B92,BCR91,D97,DV98},
image and signal processing~\cite{C03,RV91} and the fast solution of partial
differential and integral equations~\cite{D92,M08,M93}. A particularly relevant
example is provided by Alpert's polynomial multiwavelets~\cite{A93}, which extend
the seminal ideas introduced in~\cite{ABCR93}. 
These multiwavelets are locally supported, \(L^2\)-orthonormal piecewise polynomials,
see also \cite{GIvA17},
with prescribed vanishing moments and admit smooth duals, see \cite{DDLY00}.
Multivariate extensions of these techniques are typically obtained by
tensorization. 
A data-driven alternative is provided by
samplets, see~\cite{HM22,balazs2024construction,EGMQ25}. 
Samplets form a multiresolution analysis of localized discrete signed measures,
which are tailored to the underlying scattered data set
and have been used for efficient scattered data approximation and compression.
In this sense, samplets are a natural extension of the basis introduced 
in \cite{ABCR93}.
To date, however, there is no continuous-limit theory for samplets in the 
infinite data limit.

The purpose of this work is to develop such a continuous-limit theory.
To this end, we adopt a probabilistic framework in which the region containing
the scattered data sites is endowed with a probability measure and the data sites
are viewed as independent samples distributed according to this measure.
Building on the Banach frame construction of~\cite{balazs2024construction},
we interpret samplets as discretizations of discrete broken orthonormal
polynomials associated with empirical measures on a hierarchical partition.
A limit theory for samplets is then obtained by the convergence
of discrete orthonormal polynomials to their continuous counterparts.
As a byproduct, we obtain a family of non-tensorial polynomial multiwavelets.
From this perspective, samplets converge in the distributional sense to signed
measures with broken polynomial densities. In particular,
our results explicate the limiting procedure from the Nystr\"om approximation
in~\cite{ABCR93} to the multiwavelet basis in~\cite{A93} for the univariate setting.

The main contributions of the present work can be summarized as follows.
\begin{itemize}
  \item We prove the convergence of local Gram matrices and of discrete 
orthonormal polynomial families to their continuous counterparts. 

\item We show that the samplet construction is compatible with this limit 
and converges to multiwavelets defined on a hierarchical partition of the
underlying region.

\item We extend the construction beyond total-degree polynomial spaces. 
By considering downward closed index sets, we obtain, for example, 
anisotropic vanishing moments, which are suitable for high-dimensional regimes.

\item We prove that, in the special case of symmetric binary splits of 
the unit hypercube together with uniform sampling, the limiting procedure 
under consideration recovers Alpert-type multiwavelets.

\item We provide illustrative numerical studies with very large sample sizes 
  to confirm the theoretical statements.
\end{itemize}

The remainder of this article is structured as follows. 
In Section~\ref{sec:preliminaries}, we introduce the probabilistic framework and 
the fundamentals on orthogonal polynomials, empirical Gram matrices and their 
asymptotic properties.
In Section~\ref{sec:samplets}, we review the construction of samplets and 
introduce a suitable functional analytic framework.
In Section~\ref{sec:orth_pol_limits}, we develop the asymptotic analysis of 
samplets. Specifically, we study the convergence of local discrete 
orthonormal polynomials to their continuous counterparts. Further, 
we discuss the extension to downward closed index sets
and establish \(L^2\)-completeness of the limit basis.
In Section~\ref{sec:alpert_parity_complete}, we show that Alpert multiwavelets 
can be retrieved from our construction, when imposing binary partitioning 
of the unit hypercube.
The numerical experiments are collected in Section~\ref{sec:numerics},
whereas Section~\ref{sec:conclusion} states concluding remarks.

\section{Preliminaries}\label{sec:preliminaries}

\subsection{Probabilistic setting}\label{subs:prob_sett}
Let \(D\subset\Rbb^d\) be a compact set and let \(\Fcal=\Bcal(D)\) be the
Borel \(\sigma\)-algebra on \(D\). We assume that \(\Pbb\colon\Fcal\to[0,1]\)
is a probability measure that is absolutely continuous with respect to the
Lebesgue measure and has a positive Radon-Nikodym derivative. Given
independent samples \({\bs x}_1,\ldots,{\bs x}_N\in D
\) drawn with law \(\Pbb\),
we write
\[
X_N\isdef\{{\bs x}_1,\ldots,{\bs x}_N\}\subset D
\]
for the corresponding point set.

Associated to \(X_N\), we introduce the \emph{empirical measure}
\begin{equation}\label{eq:emp_meas}
	\widehat{\Pbb}_N\isdef\frac{1}{N}\sum_{i=1}^N\delta_{{\bs x}_i},
\end{equation}
where \(\delta_{\bs x}\) denotes the Dirac measure supported at
\({\bs x}\in D\). We remark that the empirical measure~\eqref{eq:emp_meas}
is a random variable due to the random selection of the set \(X_N\). 

The following classical result is known as Varadarajan's theorem,
see~\cite{D02,V58}, and guarantees the convergence of the empirical measure
towards \(\Pbb\). In what follows, we denote by \(\Ccal(D)\) the space of
continuous, and thus bounded, functions on \(D\).
\begin{theorem}\label{Varadarajan}
	The empirical measure \(\widehat{\Pbb}_N\), defined in~\eqref{eq:emp_meas},
	converges \(\Pbb\)-almost surely to \(\Pbb\), i.e.,
	\[
	\int_D f\d\widehat{\Pbb}_N\xrightarrow{N\to\infty}
	\int_D f\d\Pbb
	\quad\text{for all }f\in \Ccal(D), \ \Pbb\text{-almost surely}.
	\]
\end{theorem}

Associated with \(\Pbb\), we denote the Hilbert space
of equivalence classes of real-valued \(\Pbb\)-mea\-sur\-able and 
square-integrable functions by \(L^2(D,\Pbb)\).
Its inner product is given by
\[
(f,g)_{\Pbb} \isdef \int_{D} fg\d\Pbb,
\quad f,g\in L^2(D,\Pbb).
\]
The corresponding \(L^2\)-norm is \(\|f\|_\Pbb\isdef (f,f)_\Pbb^{1/2}\).
For a given realization \(X_N\), we
consider the \emph{empirical inner product}
\begin{equation}\label{defMuX}
	(f,g)_{\widehat{\Pbb}_N}\isdef\int_{D} fg\d\widehat{\Pbb}_N
	= \frac{1}{N}\sum_{i=1}^N f(\bs x_i)g(\bs x_i)
\end{equation}
and the associated space \(L^2(D,\widehat{\Pbb}_N)\) with norm 
\(\|f\|_{\widehat{\Pbb}_N}
\isdef (f,f)_{\widehat{\Pbb}_N}^{1/2}\).

\subsection{Orthogonal polynomials}\label{subsec:orth_pol}
We recall the notions of orthogonal and orthonormal polynomials in the
univariate setting as well as in the multivariate setting.
For all the details, we refer to~\cite{A98,C78,S67}.

We start with the univariate case, where \(D\subset\Rbb\) is a compact interval.
Then, all moments of \(\Pbb\) are finite, i.e.,
\[
\int_D x^q \d\Pbb<\infty\quad\text{for every }q\in\mathbb{N}_0.
\]
A sequence \(\{\pi_n\}_{n\in\Nbb_0}\) of polynomials is \emph{orthogonal} on
\(D\) with respect to \(\Pbb\) if
\[
(\pi_i,\pi_j)_{\Pbb}=0 \quad\text{for } i\neq j.
\]
The sequence \(\{\pi_i\}_{i\ge0}\) is \emph{orthonormal} if
additionally \(\|\pi_i\|_{\Pbb}=1\).
The sequence of orthonormal polynomials with respect to the
\((\cdot,\cdot)_{\Pbb}\)-inner product is uniquely determined
up to a sign, which can be fixed by considering monic
orthogonal polynomials.

Clearly, given the sequence of orthogonal polynomials, the corresponding
orthonormal ones can be obtained by normalization according to
\begin{equation}\label{eq:orthonormalization}
	\widehat{\pi}_i = \frac{\pi_i}{\|\pi_i\|_\Pbb}, \quad
	(\widehat{\pi}_i,\widehat{\pi}_j)_{\Pbb} = \delta_{i,j}.
\end{equation}

In the multivariate case, the orthogonal polynomials do not have 
a unique ordering with respect to their degree. Therefore, we impose
the graded lexicographic order on multi-indices and remark that other orderings
are possible.
For \(\bs\alpha,\bs\beta\in\mathbb{N}_0^d\), we write \(\bs\alpha<\bs\beta\) if
\(|\bs\alpha|<|\bs\beta|\), or if \(|\bs\alpha|=|\bs\beta|\) and there exists
\(1\le j\le d\) such that \(\alpha_i=\beta_i\) for all \(i>j\) and
\(\alpha_j<\beta_j\).

For a given multi-index
\(\bs\alpha=(\alpha_1,\ldots,\alpha_d)\in\Nbb_0^d\) with modulus
\(|\bs\alpha|\isdef\sum_{i=1}^d\alpha_i\), and
\(\bs x=[x_1,\ldots,x_d]^\intercal\in\Rbb^d\), we set
\[
\bs x^{\bs\alpha}\isdef x_1^{\alpha_1}\cdots x_d^{\alpha_d}
\]
and define the space of polynomials of total degree at most \(k\) by
\[
\Pcal_k\isdef
\operatorname{span}\{\bs x^{\bs\alpha} : \bs\alpha\in\Lambda_k\}.
\]
Therein, we denote by
\[
\Lambda_k\isdef \{\bs\alpha\in\Nbb_0^d:\ |\bs\alpha|\le k\}
\]
the set of multi-indices of modulus at most \(k\).
It is well known that the dimension of \(\Pcal_k\) is given by the
binomial coefficient
\[
\dim\Pcal_k = \binom{k+d}{d}.
\]
Correspondingly, there holds \(|\Lambda_k|=\dim\Pcal_k\).
We always consider the elements of \(\Lambda_k\) to be enumerated with 
respect to the graded
lexicographic order such that
\[
{\bs\alpha}_1<\ldots<{\bs\alpha}_{|\Lambda_k|}.
\]
Then, the associated monomials are ordered according to
\[
1=\bs x^{{\bs\alpha}_1},\ldots,\bs x^{{\bs\alpha}_{|\Lambda_k|}}=x_d^k.
\]

Analogously to the univariate case, we introduce families of orthogonal
polynomials. A family of polynomials
\(\{\pi_{\bs\alpha}\}_{\bs\alpha\in\Lambda_k}\)
is \emph{orthogonal} with respect to \(\Pbb\) if
\[
(\pi_{\bs\alpha},\pi_{\bs\beta})_{\Pbb} = 0
\quad\text{for } \bs\alpha,\bs\beta\in\Lambda_k,\ \bs\alpha\neq\bs\beta,
\]
and it is \emph{orthonormal} if, additionally,
\[
(\pi_{\bs\alpha},\pi_{\bs\alpha})_{\Pbb}=1
\quad\text{for all }\bs\alpha\in\Lambda_k.
\]
Multivariate orthogonal polynomials are obtained by applying the
Gram-Schmidt process to the monomial basis
\(\{\bs x^{{\bs\alpha}_i}\}_{i=1}^{|\Lambda_k|}\).
Then, the order of the monomials is inherited by the resulting
orthogonal family \(\{\pi_{{\bs\alpha}_i}\}_{i=1}^{|\Lambda_k|}\), which we
again consider to be monic. Thus, for each
\(i=1,\ldots,|\Lambda_k|\), we have
\[
\pi_{{\bs\alpha}_i}(\bs x)
=
\bs x^{{\bs\alpha}_i}
+\sum_{j=1}^{i-1}\ell_{i,j}\bs x^{{\bs\alpha}_j}.
\]
In particular, the Gram-Schmidt process yields a triangular change of basis from
monomials to orthogonal polynomials and an explicit representation in terms of
the Cholesky factorization of the corresponding inverse Gram matrix.
\subsection{Convergence of empirical Gram matrices}
\label{subsec:gram_moment}
We introduce the Gram matrices associated with the monomial
basis, both with respect to the measure \(\Pbb\) and with respect to the
empirical measure \(\widehat{\Pbb}_N\). These matrices encode the information
needed for the construction of the discrete orthogonal polynomials and for our
convergence analysis below. We start by collecting corresponding
non-degeneracy and convergence properties.

The monomial basis \(\{\bs x^{{\bs \alpha}_i}\}_{i=1}^{|\Lambda_k|}\) can be
orthogonalized by the Gram-Schmidt process either with respect to the
\((\cdot,\cdot)_{\widehat{\Pbb}_N}\)-inner product or with respect to the
\((\cdot,\cdot)_\Pbb\)-inner product.
As \(N\to\infty\), each element of the first family converges to the
corresponding element of the second one. Precisely, the family of discrete
monic orthogonal polynomials
\(\{\pi_{{\bs\alpha}_i}^{(N)}\}_{i=1}^{|\Lambda_k|}\), characterized by
\[
\big(\pi_{{\bs\alpha}_i}^{(N)},
\pi_{{\bs\alpha}_j}^{(N)}\big)_{\widehat{\Pbb}_N}=0
\quad\text{for } i\neq j,
\]
converges to the family of monic orthogonal polynomials 
\(\{\pi_{{\bs\alpha}_i}\}_{i=1}^{|\Lambda_k|}\),
which are characterized by
\[
(\pi_{{\bs\alpha}_i},\pi_{{\bs\alpha}_j})_\Pbb=0
\quad\text{for } i\neq j.
\]
Our proof is based on the associated Gram matrices. We
set
\[
  {\bs G}_k \isdef[g_{i,j}]_{i,j=1}^{|\Lambda_k|}
\isdef
[({\bs x}^{{\bs \alpha}_i},
{\bs x}^{{\bs \alpha}_j})_{\Pbb}]_{i,j=1}^{|\Lambda_k|}
\in\Rbb^{|\Lambda_k|\times |\Lambda_k|},
\]
as well as
\begin{equation}\label{defGkN}
  {\bs G}_k^{(N)} \isdef\big[g_{i,j}^{(N)}\big]_{i,j=1}^{|\Lambda_k|}
  \isdef
[({\bs x}^{{\bs \alpha}_i},
{\bs x}^{{\bs \alpha}_j})_{\widehat{\Pbb}_N}]_{i,j=1}^{|\Lambda_k|}
\in\Rbb^{|\Lambda_k|\times |\Lambda_k|}.
\end{equation}

\begin{remark}\label{Rem:positive_emp_gr}
	The matrices \({\bs G}_k\) are positive definite for any \(k\in\Nbb_0\).
	There holds for any \(\bs v\in \Rbb^{|\Lambda_k|}\) that
	\begin{align*}
		\bs v^\intercal{\bs G}_k\bs v
		&=\sum_{i,j=1}^{|\Lambda_k|}v_iv_j(
		\bs x^{{\bs\alpha}_i},\bs x^{{\bs\alpha}_j})_{\Pbb} 
		=\bigg(\sum_{i=1}^{|\Lambda_k|}v_i\bs x^{\bs\alpha_i},
		\sum_{j=1}^{|\Lambda_k|}v_j\bs x^{\bs\alpha_j}\bigg)_{\Pbb}\\
		&=\bigg\|\sum_{i=1}^{|\Lambda_k|}
    v_i\bs x^{\bs\alpha_i}\bigg\|_{L^2(D,\Pbb)}^2\ge 0,
	\end{align*}
    and equality holds if and only if \(\bs v=0\). Indeed, if
    \(\sum_{i=1}^{|\Lambda_k|}v_i\bs x^{\bs\alpha_i}=0\) in \(L^2(D,\Pbb)\),
    then this
    polynomial needs to vanish \(\Pbb\)-almost everywhere. Since \(\Pbb\) is assumed
    to have a positive density on \(D\), it vanishes on a set of positive
    Lebesgue measure and must thus
    be the zero polynomial. Consequently, all coefficients \(v_i\) vanish.
\end{remark}

For the Gram-Schmidt process of the discrete orthogonal polynomials, we require
the empirical Gram matrix \(\bs G_k^{(N)}\) to be positive definite. In general,
this property may fail for arbitrary point sets, in accordance with the
Mairhuber-Curtis theorem, see, e.g., \cite{W05}. However, for random samples
drawn from \(\Pbb\), the empirical Gram matrix is positive definite almost
surely if \(N\ge |\Lambda_k|\).
To establish this fact, we use the following lemma, which is essentially a
restatement of~\cite[Lemma 2.8]{W05}.
\begin{lemma}\label{lemmaWendland}
	Let \(k>0\) be an integer and \(N\ge |\Lambda_k|\). Then, there exists a
  unisolvent
	set of \(N\) points in \(D\), i.e., there exists a set
	\(X_N\isdef\{{\bs x}_1,\ldots,{\bs x}_N\}\subset D\) such that
	\(p|_{X_N}=0\) implies \(p=0\) for each \(p\in\Pcal_k\).
\end{lemma}

We have the following result on the positive definiteness of the
empirical Gram matrix \({\bs G}_k^{(N)}\).

\begin{lemma}\label{ass:nondeg}
	Let \({\bs G}_k^{(N)}\) be the empirical Gram matrix corresponding
	to the realization \(X_N\) and assume
  \(N\ge |\Lambda_k|\).
	Then, there holds
	\[
	\Pbb^{\otimes N}\big(\det{\bs G}_{k}^{(N)}=0\big)=0.
	\]
\end{lemma}
\begin{proof}
	Consider the random variable
	\[
	Q\colon D^N\to\Rbb, \quad Q(\bs x)=\det\big(\bs G_k^{(N)}(\bs x)\big).
	\]
    Then, \(Q\) is a polynomial on \(D^{N}\), and it is not the zero
    polynomial by Lemma~\ref{lemmaWendland}.
    Let \(Z_{Q}=\{\bs x\in D^{N}:Q(\bs x)=0\}\) be the 
    algebraic variety of the zeros of \(Q\).
    Since \(Q\) is not the zero polynomial, there holds
	\(
	\lambda_d^{\otimes N}(Z_{Q})=0,
	\)
	where \(\lambda_d\) is the \(d\)-dimensional Lebesgue measure,
  see~\cite{M20}. Further, since \(\Pbb\) is absolutely continuous with
  respect to \(\lambda_d\), we have that \(\Pbb^{\otimes N}\) is
  absolutely continuous with respect to \(\lambda_d^{\otimes N}\)
  and a fortiori
	\[
	\Pbb^{\otimes N}(Z_{Q})=0.
	\]
    Consequently, there holds 
    \(\Pbb^{\otimes N}\big(\det{\bs G}_k^{(N)}=0\big)=0,
    \)
    as claimed.
\end{proof}

To formulate almost-sure statements for the full sample sequence, we consider
the projective limit of the product probability spaces 
\((D^N,\mathcal F^{\otimes N},\Pbb^{\otimes N})\), endowed with the probability
measure \(\Pbb^\infty\)
characterized by
\begin{equation}
	\Pbb^\infty(A\times D\times D\times\ldots)
  =\Pbb^{\otimes N}(A), \quad A\in\mathcal{F}^{\otimes N},
\end{equation}
see, e.g., \cite[Theorem 14.36]{K20}.

\begin{corollary}\label{corollario}
	Let \(\bs x_1,\bs x_2,\bs x_3,\ldots\) be independent samples drawn with law 
  \(\Pbb\). Then, there
  \(\Pbb^\infty\)-almost surely holds \(\det\bs G_k^{(N)}\neq 0\) 
  for every \(N\ge |\Lambda_k|\)
   i.e.,
	\begin{equation}
		\Pbb^{\infty}\big(\det\bs G_k^{(N)}=0
    \textnormal{ for some }N\ge |\Lambda_k|\big)=0.
	\end{equation}
\end{corollary}
\begin{proof}
  There holds
	\begin{align}
		\big\{\bs x_1,&\bs x_2,\ldots:\text{\(\det\bs G_k^{(N)}=0\) 
    for some \(N\ge |\Lambda_k|\)}\big\}\\
		&=\bigcup_{N=|\Lambda_k|}^\infty \big\{
		\bs x_1,\bs x_2,\ldots:\text{\(\det\bs G_k^{(N)}=0\)}\big\}\\
		&=\bigcup_{N=|\Lambda_k|}^\infty \Big(\big\{\bs x_1,\ldots,\bs x_N:\det\bs
    G_k^{(N)}=0\big\}\times D\times D\times\ldots\Big).
	\end{align}
	By the definition of the projective limit, we have
	\begin{align}
		\Pbb^\infty&\big(\big\{\bs x_1,\ldots,\bs x_N:\det\bs
    G_k^{(N)}=0\big\}\times D\times D\times\ldots\big)\\
		&=\Pbb^{\otimes N}\big(\big\{\bs x_1,\ldots,\bs x_N:\det\bs G_k^{(N)}
    =0\big\}\big) =0.
	\end{align}
  The assertion follows now by the \(\sigma\)-subadditivity of \(\Pbb^\infty\)
  and the fact that the countable union of null sets is a null set.
\end{proof}

To simplify the discussion, for the rest of this work, we tacitly assume
realizations \(\{\bs x_1, \bs x_2,\ldots\}\subset D\) such that
\(\bs G_k^{(N)}\) is invertible for every \(N\ge |\Lambda_k|\). This means that
the results are valid for almost every sample
\(X=\{\bs x_1,\bs x_2,\bs x_3,\ldots\}\).

We are now in the position to prove the convergence of the empirical Gram matrix.
Since each monomial map \(\bs x\mapsto \bs x^{\bs\alpha}\),
\(\bs\alpha\in\Nbb^d_0\), is bounded and continuous on \(D\),
Theorem~\ref{Varadarajan} immediately yields the following lemma.

\begin{lemma}\label{lem:moments_conv}
	For every fixed integer \(k\ge 0\), each
  entry \(g_{i,j}^{(N)}\) of \(\bs G_k^{(N)}\)
  satisfies 
	\[
	\begin{aligned}
		g_{i,j}^{(N)}
		\xrightarrow{N\to\infty}
		g_{i,j}
    \quad\text{$\Pbb$-almost surely},\ i,j=1,\ldots,|\Lambda_k|.
	\end{aligned}
	\]
	In particular, for each such \(k\), there holds
	\(\bs G_{k}^{(N)}\xrightarrow{N\to\infty}\bs G_{k}\) 
	\(\Pbb\)-almost.
\end{lemma}
\begin{proof}
  There holds by Theorem~\ref{Varadarajan} that
	\begin{align*}
		g_{i,j}^{(N)}
		=({\bs x}^{{\bs \alpha}_i},{\bs x}^{{\bs \alpha}_j})_{\widehat{\Pbb}_N}
  =(1,{\bs x}^{{\bs\alpha}_i+{\bs\alpha}_j})_{\widehat{\Pbb}_N}
		\xrightarrow{N\to\infty}
(1,{\bs x}^{{\bs\alpha}_i+{\bs\alpha}_j})_{\Pbb}
		=({\bs x}^{{\bs \alpha}_i},{\bs x}^{{\bs \alpha}_j})_{\Pbb}
=
		g_{i,j}
	\end{align*}
\(\Pbb\)-almost surely. The convergence of the empirical Gram matrix follows
from the convergence of each of its entries.
\end{proof}

\subsection{Convergence of discrete orthogonal polynomials}
In this subsection, we prove the convergence of the discrete orthogonal
polynomials associated to \(\widehat{\Pbb}_N\) towards their counterparts
associated to \(\Pbb\). 

\begin{proposition}\label{lem:ort_pol_conv_multi}
	For each \(\bs\alpha\in\Lambda_k\), there holds
	\(
	\pi_{\bs\alpha}^{(N)}\xrightarrow{N\to\infty}\pi_{\bs\alpha},
	\)
	uniformly in \(\bs x\in D\), \(\Pbb\)-almost surely.
\end{proposition}

\begin{proof}
  Fix \({\bs\alpha}_i\in\Lambda_k\).
  By the triangular structure induced by the graded lexicographic ordering,
  we can write
	\[
	\pi_{{\bs\alpha}_i}(\bs x)
	=
	\bs x^{{\bs\alpha}_i}
	+\sum_{j=1}^{i-1}\ell_{i,j}\bs x^{{\bs\alpha}_j},
	\quad
	\pi_{{\bs\alpha}_i}^{(N)}(\bs x)
	=
	\bs x^{{\bs\alpha}_i}
	+\sum_{j=1}^{i-1}\ell_{i,j}^{(N)}\bs x^{{\bs\alpha}_j}.
	\]
	The orthogonality conditions
	\((\pi_{{\bs\alpha}_i},\bs x^{{\bs\alpha}_j})_\Pbb=0\) 
  for \(j=1,\ldots,i-1\) are equivalently formulated by the linear system
	\[
    \bs G_{i-1} \bs \ell_i=- \bs g_i,
    \]
    where 
    \[
      \bs G_{i-1}
      =[(\bs x^{{\bs\alpha}_j},\bs x^{{\bs\alpha}_r})_\Pbb]_{j,r=1}^{i-1},\quad 
  \bs\ell_i=[\ell_{i,1},\ldots,\ell_{i,i-1}]^\intercal,\quad 
  \bs g_i=[(\bs x^{{\bs\alpha}_i},\bs x^{{\bs\alpha}_j})_\Pbb]_{j=1}^{i-1}.
	\]
	Likewise, we obtain a linear system with respect to \(\widehat{\Pbb}_N\) given
  by
	\[
    \bs G_{i-1}^{(N)} \bs \ell_i^{(N)}=- \bs g_i^{(N)},
	\]
  where all quantities are obtained by replacing \(\Pbb\) with its empirical
  counterpart.
 The convergence of \(\bs G_{i-1}^{(N)}\) to \(\bs G_{i-1}\) is shown in
 Lemma~\ref{lem:moments_conv}, while the \(\Pbb\)-almost sure convergence
 of \(\bs g_i^{(N)}\xrightarrow{N\to\infty}\bs g_i\) follows in a similar fashion.

 By assumption, \(\bs G_i\) is invertible. Therefore, by
 the continuity of the
 matrix inversion on \(\mathrm{GL}_{i-1}(\mathbb{R})\), we have
	\[
    \bs \ell_i^{(N)}=-\big(\bs G_{i-1}^{(N)}\big)^{-1} \bs g_i^{(N)}
			\xrightarrow{N\to\infty}
  -\bs G_{i-1}^{-1} \bs g_i
	=\bs \ell_i
	\quad\text{$\Pbb$-almost surely}.
	\]
  This proves the $\Pbb$-almost sure convergence	
  \(\pi_{\bs\alpha}^{(N)}\xrightarrow{N\to\infty}\pi_{\bs\alpha}\).

	Finally, since all monomials are bounded by a common constant, i.e.,
  \(|\bs x^{{\bs\alpha}_j}|\le C\) for all \(\bs x\in D\)
  and all \(j\leq i-1\) and some \(C>0\), we arrive at
	\[
	\sup_{\bs x\in D}
  \big|\pi_{{\bs\alpha}_i}^{(N)}(\bs x)-\pi_{{\bs\alpha}_i}(\bs x)\big|
	\le C \sum_{j=1}^{i-1}\big|\ell_{i,j}^{(N)}-\ell_{i,j}\big|
			\xrightarrow{N\to\infty} 0
	\quad\text{$\Pbb$-almost surely}.
	\]
	This proves the uniform convergence.
\end{proof}

Since all polynomials are bounded on \(D\), we obtain the \(L^2(D,\Pbb)\) 
convergence as a direct consequence.
\begin{corollary}\label{cor:norm_conv_multi}
	Under the assumptions of Proposition~\ref{lem:ort_pol_conv_multi}, 
  for each \(\bs\alpha\in\Lambda_k\), there holds
	\[
	\lim_{N \to \infty}\big\|\pi_{\bs\alpha}^{(N)}\big\|_{\widehat{\Pbb}_N}
	=
	\|\pi_{\bs\alpha}\|_\Pbb
	\quad\text{
	$\Pbb$-almost surely}.
	\]
\end{corollary}

\begin{remark}
Together with Proposition~\ref{lem:ort_pol_conv_multi} and 
Corollary~\ref{cor:norm_conv_multi}, we obtain the uniform convergence 
of discrete orthonormal polynomials to their continuous counterparts after 
normalization.
\end{remark}

A similar statement remains valid when the discrete orthogonal polynomials are
constructed from one realization, while the inner product is induced by
the empirical measure associated with another realization, both drawn
from the same probability measure \(\Pbb\). Such statement is useful when
comparing polynomials across different sample sets or resolutions.

\begin{proposition}\label{lem:two_empirical}
Let
	\(
	X_N=\{\bs x_1,\ldots,\bs x_N\}\subset D,
	\quad
	Y_M=\{\bs y_1,\ldots,\bs y_M\}\subset D,
	\)
	be two independent sets of independent samples drawn with common law
	\(\Pbb\). Let 
    \(\widehat{\Pbb}_M^Y\) be the
    empirical measure associated with $Y_M$.
    Then, for any fixed
	\(\bs\alpha,\bs\beta\in\Lambda_k\), there holds
	\[
	\lim_{M \to \infty} \lim_{N \to \infty}
	\big(\pi_{\bs\alpha}^{(N)}, \pi_{\bs\beta}^{(N)}\big)_{\widehat{\Pbb}_M^Y}
	=
	\lim_{N \to \infty}\lim_{M \to \infty}
	\big(\pi_{\bs\alpha}^{(N)}, \pi_{\bs\beta}^{(N)}\big)_{\widehat{\Pbb}_M^Y}
	=
	(\pi_{\bs\alpha}, \pi_{\bs\beta})_{\Pbb}
	\]
	$\Pbb$-almost surely.
	In particular, if \(\{\pi_{\bs\gamma}\}_{\bs\gamma\in\Lambda_k}\) is
  orthonormal
	with respect to \(\Pbb\), then the limit equals
  \(\delta_{\bs\alpha,\bs\beta}\).
\end{proposition}

\begin{proof}
	By Proposition~\ref{lem:ort_pol_conv_multi}, we have
\(\pi_{\bs\alpha}^{(N)}\xrightarrow{N\to\infty}\pi_{\bs\alpha}\) and
\(\pi_{\bs\beta}^{(N)}\xrightarrow{N\to\infty}\pi_{\bs\beta}\) $\Pbb$-almost surely
uniformly on \(D\).
Fix \(M\) and let \(N\to\infty\). By uniform convergence and boundedness,
	\[
    \big(\pi_{\bs\alpha}^{(N)},\pi_{\bs\beta}^{(N)}\big)_
    {\widehat{\Pbb}_M^Y}
    		\xrightarrow{N\to\infty} (\pi_{\bs\alpha},\pi_{\bs\beta})_{\widehat{\Pbb}_M^Y}.
	\]
    Now let \(M\to\infty\). Since \(\pi_{\bs\alpha}\pi_{\bs\beta}\) is
    continuous and bounded on \(D\), weak convergence of \(\widehat{\Pbb}_M^Y\)
    to \(\Pbb\) yields
    \[
    (\pi_{\bs\alpha},\pi_{\bs\beta})_{\widehat{\Pbb}_M^Y}\xrightarrow{M\to\infty}
    (\pi_{\bs\alpha},\pi_{\bs\beta})_\Pbb
    \]
    $\Pbb$-almost surely.
	This gives the first iterated limit. The second iterated limit is obtained 
	by reversing the order of the two steps. By observing that for fixed 
	\(N\) the function \(\pi_{\bs\alpha}^{(N)}\pi_{\bs\beta}^{(N)}\) is
	continuous and bounded, we obtain
	\[
	\big(\pi_{\bs\alpha}^{(N)},\pi_{\bs\beta}^{(N)}\big)_{\widehat{\Pbb}_M^Y}
  		\xrightarrow{M\to\infty}\big(\pi_{\bs\alpha}^{(N)},\pi_{\bs\beta}^{(N)}\big)_{\Pbb},
	\]
	and therefore
\[
\lim_{N\to\infty}\lim_{M\to\infty}
\big(\pi_{\bs\alpha}^{(N)},\pi_{\bs\beta}^{(N)}\big)_{\widehat{\Pbb}_M^Y}
=
(\pi_{\bs\alpha},\pi_{\bs\beta})_\Pbb
\]
by uniform convergence of \(\pi_{\bs\alpha}^{(N)}\xrightarrow{N\to\infty}\pi_{\bs\alpha}\) and
\(\pi_{\bs\beta}^{(N)}\xrightarrow{N\to\infty}\pi_{\bs\beta}\) on \(D\).
\end{proof}

\subsection{Rates and discrepancy effects}\label{subsec:rate}
The $\Pbb$-almost sure limits established above are qualitative.
In applications, one is often interested in quantitative convergence rates,
which depend on the sampling quality. In this regard, especially for
low-discrepancy point
sets, one can derive deterministic bounds in terms of the
\emph{star discrepancy} \(\mathrm{disc}_N^*\) via
Koksma-Hlawka estimates, under finite Hardy-Krause variation assumptions for
the relevant integrands. Since these bounds are not used later, we defer the
precise statements and proofs to Appendix~\ref{app:rates_discrepancy}.

\section{Samplets}\label{sec:samplets}
\subsection{Banach frame setting}\label{subsec:BFS}
Samplets are a multiresolution analysis of localized discrete signed measures
and can be considered a discrete version of wavelets.
We recall here their construction as introduced in~\cite{HM22}.
The reader will notice that, in contrast to the 
existing literature, in this work samplets are rescaled by the harmless 
factor \(1/N\). The presence of this factor becomes relevant when varying the 
number of samples \(N\), especially when resorting to results such as Varadarajan's 
theorem and the convergence of empirical moments established in 
Section~\ref{sec:preliminaries}.

As before, let \(X_N=\{\bs x_1,\ldots,\bs x_N\}\subset D\) be an independent sample
drawn with law \(\Pbb\)
and let \(\widehat{\Pbb}_N\) be the associated empirical measure
defined in~\eqref{eq:emp_meas}. We set
\[
\Xcal' \isdef \spn\{ \delta_{\bs x_1}, \ldots, \delta_{\bs x_N} \}.
\]
Elements of \(\Xcal'\) are understood as finitely supported distributions on
\(D\), acting on functions \(f\in\mathcal{C}(D)\) by
\begin{equation}\label{eq:defActionu}
\bigg(\sum_{i=1}^N u_i\delta_{\bs x_i}\bigg)(f)=\sum_{i=1}^N u_i f(\bs x_i).
\end{equation}
To obtain suitable identifications between vectors, functions and distributions, 
we employ the framework proposed in~\cite{balazs2024construction}. Concretely,
we consider the synthesis and analysis operators
    \begin{align}
    &S_N\colon\Rbb^N\to\Xcal',\quad [c_i]_{i=1}^N\mapsto
    \sqrt N\sum_{i=1}^N c_i\delta_{\bs x_i} \in \Xcal', \\
    \label{defSstar}
    &S^\ast_N\colon \Ccal(D)\to\Rbb^N,\quad f\mapsto
    \frac{1}{\sqrt N}[f(\bs x_i)]_{i=1}^N\in\Rbb^N.
    \end{align}
These operators are dual with respect to the specific duality pairing
\[
  \langle\cdot,\cdot\rangle_{\Xcal' \times\Ccal(D)} \colon\Xcal' \times\Ccal(D) 
  \to \Rbb,
  \qquad
  \bigg\langle\sum_{i=1}^N c_i\delta_{{\bs x}_i},
  f\bigg\rangle_{\Xcal' \times\Ccal(D)} 
  \isdef
  \frac{1}{N} \sum_{i=1}^N c_i f({\bs x}_i),
\]
 Concatenating \(S_N\) and \(S_N^\ast\) gives rise to the frame operator
 \begin{equation}\label{eq:FrameOp}
 F_N\isdef S_NS_N^\ast\colon\Ccal(D)\to\Xcal',\quad f\mapsto
 \sum_{i=1}^Nf({\bs x}_i)\delta_{{\bs x}_i}.
 \end{equation}
    The operator $S_N$ is obviously invertible with inverse
    \begin{equation}
        S^{-1}_N\colon\Xcal'\to\Rbb^N,\quad
        \sum_{i=1}^Nc_i\delta_{\bs x_i}\mapsto\frac{1}{\sqrt N}[c_i]_{i=1}^N\in\Rbb^N.
    \end{equation}
    Using the operator \(S^{-1}_N\), we transport the standard inner product on
    \(\Rbb^N\) to \(\Xcal'\), i.e.,
\begin{equation}\label{rescaledIP}
	(u, v )_{\Xcal'} \isdef \big(S^{-1}_Nu,S^{-1}_Nv\big)_{\Rbb^N}, \quad u,v\in\Xcal'.
 \end{equation}
 Concretely, we obtain for \(	u = \sum_{i=1}^{N} u_i \delta_{\bs x_i},\
	v = \sum_{i=1}^{N} v_i \delta_{\bs x_i}\) that
 \begin{equation}
   (u, v )_{\Xcal'} = \frac{1}{N}\sum_{i=1}^{N} u_i v_i.
\end{equation}
In particular, this implies that 
\(\{\delta_{{\bs x}_1},\ldots,\delta_{{\bs x}_N}\}\) forms an orthogonal basis
in \(\Xcal'\). The corresponding orthonormal basis is given by
\(\big\{\widehat\delta_{{\bs x}_1},\ldots,\widehat\delta_{{\bs x}_N}\big\}\), where
\[
\widehat\delta_{{\bs x}_i}\isdef\sqrt{N}\delta_{{\bs x}_i}
\quad\text{for }i=1,\ldots,N.
\]

Similarly to the \(\Xcal'\)-inner product, we can identify 
the $\widehat\Pbb_N$-inner product defined in~\eqref{defMuX}
 as
\begin{equation}
    (f,g)_{\widehat\Pbb_N}=\big(S^\ast_N f,S^\ast_N g\big)_{\Rbb^N}=
    \langle F_N f,g\rangle_{\Xcal'\times\Ccal(D)}
    =\frac 1 N (F_Nf)(g), \quad\text{for all }
    f,g\in\Ccal(D).
\end{equation}
In particular, we find
\begin{equation}
    (f,g)_{\widehat\Pbb_N}=\big(F_N f,F_N g\big)_{\Xcal'}.
\end{equation}
    
\subsection{Samplet construction}
Samplets are a form of multiresolution analysis on \(\Xcal'\).
They are constructed using a nested sequence of subspaces
\begin{equation}\label{eq:nested_seq}
	\Xcal'_0 \subset \Xcal'_1 \subset \cdots \subset \Xcal'_J \isdef \Xcal',
\end{equation}
where \(\Xcal'_j \isdef \spn(\bs \Phi_j)\), and \(\bs \Phi_j 
\isdef \{\varphi_{j,\ell}\}_\ell\) is an orthonormal basis with respect
to \((\cdot,\cdot)_{\Xcal'}\).
In our case, each scaling distribution \(\varphi_{j,\ell}\) is a 
linear combination of Dirac measures.
We may orthogonally decompose each \(\Xcal'_{j+1}\) as
\begin{equation}\label{eq:nested_seq2}
	\Xcal'_{j+1} = \Xcal'_j \overset{\perp}{\oplus} \Scal_j',
\end{equation}
and denote by \(\bs \Sigma_j \isdef \{\sigma_{j,\ell}\}_\ell\) the orthonormal 
basis of each \emph{detail space} \(\Scal_j'\), which is called the 
\emph{samplet basis}. Iterating yields
\[
\bs \Sigma_J = \bs \Phi_0 \cup \bigcup_{j=0}^{J-1} \bs \Sigma_j,
\]
which forms a basis of \(\Xcal'\). 
To promote data compression, 
the distributions \(\sigma_{j,\ell}\) are constructed to 
satisfy the vanishing moment condition
\[
\sigma_{j,\ell}(p) = 0
\quad \text{for all } p \in \Pcal_k,
\]
for fixed $k\in\Nbb_0$.

The simplest construction of the multiresolution analysis~\eqref{eq:nested_seq}
is based on a hierarchical clustering of the set \(X_N\), which amounts to a
clustering of the Dirac measures in \(\Xcal'\) with respect to their supports.
To this end, we introduce the notion of a \emph{cluster tree}.

\begin{definition}\label{def:cluster-tree_revised}
	Let \(\Tcal = (V, E)\) be a tree with vertices \(V\) and edges \(E\).
	We denote the set of leaves of \(\Tcal\) by
	\(
	\Lcal(\Tcal) \isdef \{\tau \in V : \tau \text{ has no children}\}.
	\)
	The tree \(\Tcal\) is a \emph{cluster tree} for \(X_N\) if \(X_N\) is the
  root of \(\Tcal\) and each vertex \(\tau \in V \setminus \Lcal(\Tcal)\)
  is the disjoint union of its children.
	The level \(j_\tau\) of \(\tau\) is its distance from the root
  and the depth of the tree is denoted by \(J \isdef \max_{\tau \in \Tcal} j_\tau\).
  Moreover we refer to the set of clusters at level \(j\) as
	\(
	\Tcal_j \isdef \{\tau \in \Tcal : j_\tau = j\}.
	\)
\end{definition}

For simplicity, we exclusively consider binary cluster trees here, i.e., every
non-leaf cluster has exactly two children, and remark that different
constructions are applicable with the straightforward modifications.

With the cluster tree at our disposal, we turn to the construction of the nested
sequence in~\eqref{eq:nested_seq}.
Each scale \(\Xcal'_j\), associated with the corresponding level \(\Tcal_j\),
is constructed recursively from the information contained in \(\Xcal'_{j+1}\).
We denote the set of basis elements, which are
supported in the cluster \(\tau\) at level \(j\leq J\) by \({\bs\Phi}^{\tau,(N)}_j\) 
and set
 \(\bs\Phi_j^{(N)}=\big\{{\bs\Phi}^{\tau,(N)}_j\big\}_{\tau\in\Tcal_j}\). Note that
 here and
 in what follows,
 we make the dependence of the distributions
 on the sample explicit by using the superscript \((N)\).
By a slight abuse of notation, we identify \({\bs\Phi}^{\tau,(N)}_j\) and
\(\bs\Phi_j^{(N)}\), respectively, as row vectors,
where each entry corresponds to a basis element.

For each \(j\leq J\) and each \(\tau\in\Tcal\), we introduce the 
\emph{moment matrix}
\begin{equation}\label{momentMatrix}
	{\bs M}_{j+1}^\tau
	\isdef
  \left[{\bs\Phi}^{\tau,(N)}_{j+1}(\bs x^{\bs{\alpha}_i})\right]_{i=1}^{|\Lambda_k|}
	\in \mathbb{R}^{|\Lambda_k| \times |\bs\Phi_{j+1}^\tau|},
\end{equation}
where we set
\(\bs\Phi_{J+1}^{\tau,(N)}\isdef
\big\{\widehat\delta_{{\bs x}_i}\big\}_{{\bs x}_i\in\tau}\).
Next, we consider the QR decomposition of the transpose moment matrix given by
\[
(\bs{M}^\tau_{j+1})^{\intercal} = \bs{Q}^\tau_j \bs{R}^\tau_j,
\]
where \(\bs{Q}^\tau_j\) is orthogonal and \(\bs{R}^\tau_j\) is upper triangular.
The columns of the matrix \({\bs Q}_j^\tau\) correspond to the filter coefficients,
which ensure orthonormality with respect to the \((\cdot,\cdot)_{\Xcal'}\)-inner product
 as well as
the vanishing moment condition.
The refinement relation in cluster \(\tau\) at level \(j\leq J\) is hence given by
\begin{equation}
\big[ {\bs \Phi}_j^{\tau,(N)}, {\bs \Sigma}_j^{\tau,(N)}\big]
\isdef {\bs \Phi}_{j+1}^{\tau,(N)} {\bs Q}_j^\tau.
\end{equation}
By splitting the matrix \({\bs Q}_j^\tau\) according to
\({\bs Q}_j^\tau=\big[ {\bs Q}_{j,\Phi}^\tau, {\bs Q}_{j,\Sigma}^\tau \big]\),
where \({\bs Q}_{j,\Phi}^\tau\) contains the first \(|\Lambda_k|\) columns and
\({\bs Q}_{j,\Sigma}^\tau\) the remaining ones,
we can write
\begin{equation}\label{eq:samplets_refinement}
\big[ {\bs \Phi}_j^{\tau,(N)}, {\bs \Sigma}_j^{\tau,(N)}\big]
= {\bs \Phi}_{j+1}^{\tau,(N)}
[ {\bs Q}_{j,\Phi}^\tau, {\bs Q}_{j,\Sigma}^\tau ].
\end{equation}
Therefore, the \emph{scaling distributions} at level \(j\) supported on \(\tau\)
are given by
\[
\varphi_{j,\ell}^{\tau,(N)}=\sum_{i=1}^{|\bs\Phi_{j+1}^{\tau,(N)}|}
(\bs Q_{j,\Phi}^\tau)_{i,\ell}\varphi_{j+1,i}^{\tau,(N)},
\]
while the \emph{samplets} at the same level supported on \(\tau\) are given by
\[
  \sigma_{j,\ell}^{\tau,(N)}=\sum_{i=1}^{|\bs\Phi_{j+1}^{\tau,(N)}|}
  (\bs Q_{j,\Sigma}^\tau)_{i,\ell}\varphi_{j+1,i}^{\tau,(N)}.
\]

The associated moment matrix particularly satisfies
\[
  \Big[ \bs{\Phi}^{\tau,(N)}_j(\bs x^{\bs{\alpha}_i}),
  \bs{\Sigma}^{\tau,(N)}_j(\bs x^{\bs{\alpha}_i}) \Big]_{i=1}^{|\Lambda_k|}
=
\bs{M}^\tau_{j+1}[ \bs{Q}^\tau_{j,\Phi},\bs{Q}^\tau_{j,\Sigma}]
=
(\bs{R}^\tau_j)^{\intercal}.
\]

As \((\bs{R}^\tau_j)^{\intercal}\) is a lower triangular matrix, the first \(\ell-1\) 
entries in its \(\ell\)-th column are zero. This corresponds to \(\ell-1\) vanishing
moments for the \(\ell\)-th distribution generated by the transformation
\(\bs Q^{\tau}_{j} = [\bs{Q}^\tau_{j,\Phi},\bs{Q}^\tau_{j,\Sigma}]\).
Therefore, considering the first \(|\Lambda_k|\) distributions as scaling 
distributions,
we obtain indeed samplets with at least vanishing moments of degree \(k\).

We collect the filter coefficients of all samplets and of the scaling distributions at 
level \(0\) in the columns of the matrix \({\bs U}\in\Rbb^{N\times N}\) with respect to
some global level-wise ordering, such that 
\begin{equation}\label{eq:sampGloRep}
  \sigma_i^{(N)}=\big[\widehat\delta_{{\bs x}_1},\ldots,
\widehat\delta_{{\bs x}_N}\big]{\bs u}_i,
\end{equation}
where \({\bs u}_i\) is the \(i\)-th column of \({\bs U}\). In particular, there holds
\({\bs U}{\bs U}^\intercal={\bs U}^\intercal{\bs U}={\bs I}\), since the supports of 
any two columns are either disjoint or their entries are obtained from a sequence 
of orthogonal matrix products, rendering them orthogonal. Therefore, we obtain
\begin{align*}
\Big(\big(\big[\widehat\delta_{{\bs x}_1},\ldots,
\widehat\delta_{{\bs x}_N}\big]{\bs U}\big)^\intercal,
\big[\widehat\delta_{{\bs x}_1},\ldots,
\widehat\delta_{{\bs x}_N}\big]{\bs U}\Big)_{\Xcal'}
={\bs U}^\intercal\big[\big(\widehat\delta_{{\bs x}_i},
    \widehat\delta_{{\bs x}_j})_{\Xcal'}]_{i,j=1}^N
{\bs U}={\bs U}^\intercal{\bs I}{\bs U}={\bs I}
\end{align*}
which shows the orthogonality of the samplet basis.
With the aid of the matrix \({\bs U}\),
the vanishing moment condition can be written as
\[
  \sigma_{i}^{(N)}(p) = ({\bs u}_i,S^\ast_Np)_{\Rbb^N}=0
\quad \text{for all } p \in \Pcal_k.
\]

\section{Asymptotics of samplets and broken polynomials}
\label{sec:orth_pol_limits}
\subsection{Convergence of local Gram matrices}
The asymptotic analysis of samplets requires the convergence of discrete
orthonormal polynomials to their continuous counterparts. 
To study this convergence, we assume that the cluster tree underlying the
samplets is generated from a hierarchical dyadic partition of the set \(D\).
In the present
setting, however, the relevant hierarchy is not introduced directly from the
sample set \(X_N\), but rather from a geometric partition of the set \(D\)
itself.
At each level \(j\), we consider a family of pairwise
disjoint Borel sets with positive \(\Pbb\) measure
\[
\mathcal D_j=\{D_\tau:\tau\in\Tcal_j\},
\]
whose union is \(D\), and such that each set \(D_\tau\in\mathcal D_j\),
\(j<J\), is the disjoint union of its two children
\(D_{\tau_1},D_{\tau_2}\in\mathcal D_{j+1}\). 
In practice, one may think of \(D_\tau\) as the intersection 
of \(D\) with one of the axis-aligned boxes obtained by recursively b
isecting the bounding box of \(D\). Once the partition is fixed, one draws the
sample \(X_N=\{\bs x_1,\ldots,\bs x_N\}\subset D\).
Then, partitioning \(X_N\)
with respect to the hierarchy of partitioning sets, it is clear that 
we obtain a cluster tree for
\(X_N\). Vice versa, the sets \(\{D_\tau\}_{\tau\in\Tcal}\) 
serve as the geometric bounding regions of the resulting cluster tree.

For the asymptotic analysis, we fix a node \(\tau\in\Tcal_J\) and
observe that \(|\tau|\xrightarrow{N\to\infty}\infty\), since we assume
\(\Pbb(D_\tau)>0\).
Next, let \(\Pbb|_{D_\tau}\) denote the restriction of \(\Pbb\) to \(D_\tau\)
and \(\widehat{\Pbb}_{N}^{\tau}\) the cluster-restricted empirical measure
defined as
\begin{equation}\label{eq:mu_N_nu}
	\widehat{\Pbb}_{N}^{\tau}
	\isdef
	\frac1N\sum_{\bs x_i\in \tau}\delta_{\bs x_i}.
\end{equation}
Then, as a consequence of Theorem~\ref{Varadarajan}, there holds
\begin{equation}\label{eq:cluster_assumption}
	\widehat{\Pbb}_{N}^{\tau}\xrightharpoonup{N\to\infty} \Pbb|_{D_\tau}.
\end{equation}
It is clear that \(\widehat{\Pbb}_{N}^{\tau}\) is in general no
probability measure, since
\[
\widehat{\Pbb}_{N}^{\tau}(D)
=
\widehat{\Pbb}_{N}^{\tau}(D_\tau)
=
\frac{|\tau|}{N}.
\]
Analogously to~\eqref{defMuX}, for \(f,g\in\Ccal(D)\), the 
\(\widehat{\Pbb}_{N}^{\tau}\)-inner product is given by
\begin{equation}\label{defmuNnu}
	(f,g)_{\widehat{\Pbb}_{N}^{\tau}}
	=
	\frac1N\sum_{\bs x_i \in \tau} f(\bs x_i)g(\bs x_i)
\end{equation}
and we denote the corresponding norm by
\[
\|f\|_{\widehat{\Pbb}_{N}^{\tau}}
\isdef
(f,f)_{\widehat{\Pbb}_{N}^{\tau}}^{1/2}.
\]
Especially, there now holds
\[
(1,1)_{\widehat{\Pbb}_{N}^{\tau}}=\frac{|\tau|}{N}.
\]
It will be convenient to keep the global
identification of \(\Xcal'\) with \(\Rbb^N\)
introduced in Subsection~\ref{subsec:BFS}. Letting
\(\mathbbm{1}_{D_\tau}\) denote the
indicator function of \(D_\tau\), we have
\[
(f,g)_{\widehat{\Pbb}_{N}^{\tau}}
=
(f\mathbbm{1}_{D_\tau},g\mathbbm{1}_{D_\tau})_{\widehat{\Pbb}_N}
=
\big(S_N^\ast(f\mathbbm{1}_{D_\tau}),
S_N^\ast(g\mathbbm{1}_{D_\tau})\big)_{\Rbb^N}.
\]
This way, all cluster-wise defined quantities are considered to be
extended by zero to the ambient space. In the same spirit, we introduce
the restricted polynomial space
\begin{equation}\label{eq:defPkDtau}
	\Pcal_k^\tau\isdef \{p\mathbbm{1}_{D_\tau}:p\in\Pcal_k\}.
\end{equation}
All orthogonality and orthonormality statements below are understood with
respect to either \((\cdot,\cdot)_{\widehat{\Pbb}_{N}^{\tau}}\) on
\(\Pcal_k^\tau\) or \((\cdot,\cdot)_{\Pbb|_{D_\tau}}\) on
\(L^2(D_\tau,\Pbb|_{D_\tau})\).

In line with Lemma~\ref{ass:nondeg}, we have the following non-degeneracy
statement.
\begin{lemma}\label{lemma:constpol}
	Assume that \(|\tau|\ge |\Lambda_k|\). Then, the map
	\[
	\Pcal_k^\tau\to\Rbb^N,
	\quad
	p\mapsto S_N^\ast p,
	\]
	is \(\Pbb^{\otimes N}\)-almost surely injective on \(\Pcal_k^\tau\). 
  In particular, any \(p\in\Pcal_k^\tau\) is uniquely determined by its
  values at the sample points in \(\tau\).
\end{lemma}

\begin{proof}
  Let \(p\in\Pcal_k^\tau\) and assume that \(S_N^\ast p={\bs 0}\in\Rbb^N\).
  Choose a polynomial \(q\in\Pcal_k\) such that
	\(p=q|_{\tau}\). Writing
	\[
	q(\bs x)=\sum_{j=1}^{|\Lambda_k|} c_j \bs x^{\bs\alpha_j}
	\]
  and noticing \(p({\bs x}_i)=q({\bs x}_i)=0\) for all \({\bs x}_i\in\tau\),
  we obtain the linear system
  \begin{equation}\label{eq:homVandermonde}
    \bs V_\tau \bs c={\bs 0},
  \end{equation}
  where \(\bs c=[c_j]_{j=1}^{|\Lambda_k|}\in\Rbb^{|\Lambda_k|}\) and
	\(\bs V_\tau\in\Rbb^{|\tau|\times |\Lambda_k|}\) is the generalized
  Vandermonde matrix associated with the points in \(\tau\), i.e.,
	\[
    \bs V_\tau=[\bs x_{i}^{\bs\alpha_j}
    ]_{\gfrac{i=1,\ldots|\tau|}{j=1,\ldots,|\Lambda_k|}}.
	\]
	The condition~\eqref{eq:homVandermonde} implies that
  \[
    \bs G_{k,\tau}^{(N)}\bs c={\bs 0},
	\quad
	\bs G_{k,\tau}^{(N)}\isdef \frac1N\bs V_\tau^\intercal \bs V_\tau.
	\]
	By the same argument as in Lemma~\ref{ass:nondeg}, the matrix
	\(\bs G_{k,\tau}^{(N)}\) is almost surely invertible. Hence,
  there holds \(\bs c={\bs 0}\), such that
	\(q=0\) and therefore \(p=0\) on \(D_\tau\).
\end{proof}

We immediately obtain the cluster-wise analogue of
Corollary~\ref{corollario}.

\begin{corollary}\label{corollario2}
	Assume that \(|\tau|\ge |\Lambda_k|\) for every leaf node
  \(\tau\in\Lcal(\Tcal)\). Then
	\[
	\Pbb^{\otimes N}\big(\text{\(\bs G_{k,\tau}^{(N)}\) \textnormal{is invertible
  for every leaf} \(\tau\)}\big)=1.
	\]
	Consequently,
	\[
	\Pbb^\infty\big(\textnormal{for every \(N\) sufficiently large, 
  \(\bs G_{k,\tau}^{(N)}\) is invertible for every leaf \(\tau\)}\big)=1.
	\]
\end{corollary}

In view of the preceding result, throughout the remainder of this section,
we may restrict ourselves to realizations for which all leaf-wise Gram
matrices are invertible.

\subsection{Samplets and broken orthonormal polynomials}
\label{subsec:qr_moments}
In the samplet construction, at the leaves of the cluster tree,
the scaling distributions are determined by a QR decomposition 
of the transpose of the moment matrices~\eqref{momentMatrix},
which coincide with the
generalized Vandermonde matrices
	\[
    \bs V_\tau=[\bs x_{i}^{\bs\alpha_j}
    ]_{\gfrac{i=1,\ldots|\tau|}{j=1,\ldots,|\Lambda_k|}},
    \quad\tau\in\Lcal(\Tcal).
	\]

As we only consider scaling distributions at the moment,
we may particularly start from the thin QR decomposition
\(	\bs V_\tau=\bs Q_\tau \bs R_\tau,
\) where \({\bs R}_\tau\in\Rbb^{|\Lambda_k|\times|\Lambda_k|}\)
is upper triangular and has full rank. Consequently, there
holds \(\bs V_\tau{\bs R}_\tau^{-1}={\bs Q}_\tau\), which yields
that the basis
\([{\bs x}^{\bs\alpha_1},\ldots,
{\bs x}^{\bs\alpha_{|\Lambda_k|}}]{\bs R}_\tau^{-1}\) 
is orthogonal, since
\begin{align*}
\Big(
  \big([{\bs x}^{\bs\alpha_1},\ldots,
{\bs x}^{\bs\alpha_{|\Lambda_k|}}]{\bs R}_\tau^{-1}\big)^\intercal&,
[{\bs x}^{\bs\alpha_1},\ldots,
{\bs x}^{\bs\alpha_{|\Lambda_k|}}]
{\bs R}_\tau^{-1}\Big)_{\widehat{\Pbb}_N^\tau}\\
&=\frac{1}{N}{\bs R}_\tau^{-\intercal}{\bs V}_\tau^\intercal
{\bs V}_\tau{\bs R}_\tau^{-1}
=\frac{1}{N}{\bs Q}_\tau^\intercal{\bs Q}_\tau
=\frac{1}{N}{\bs I}.
\end{align*}
The basis becomes orthonormal if we apply the change of basis
\(\sqrt{N}{\bs R}_\tau^{-1}\) instead, i.e.,
\[
\big[\widehat\pi_{1}^{\tau,(N)},\ldots,\widehat\pi_{|\Lambda_k|}^{\tau,(N)}\big] 
=\sqrt{N}[{\bs x}^{\bs\alpha_1},\ldots,{\bs x}^{\bs\alpha_{|\Lambda_k|}}]
{\bs R}_\tau^{-1}.
\]

Considering the restrictions of
\(\widehat\pi_{1}^{\tau,(N)},\ldots,\widehat\pi_{|\Lambda_k|}^{\tau,(N)}\) 
to \(\Pcal_k^\tau\),
we particularly find
\begin{align*}
  {\bs\Phi}_J^{\tau,(N)}=F_N\big[\widehat\pi_{1}^{\tau,(N)},
  \ldots,\widehat\pi_{|\Lambda_k|}^{\tau,(N)}\big]  
  =\big[\widehat\delta_{{\bs x}_{i_1}},\ldots,
  \widehat\delta_{{\bs x}_{i_{|\tau|}}}\big]{\bs Q}_\tau
\end{align*}
with
\[
  \Big(\big({\bs\Phi}_J^{\tau,(N)}\big)^\intercal,
  {\bs\Phi}_J^{\tau,(N)}\Big)_{\Xcal'}
  ={\bs I}.
\]
This means that the scaling distributions are obtained by applying
the frame operator \(F_N\) to the discrete orthonormal polynomials
in \(\Pcal_k^\tau\). Vice versa, the discrete orthonormal polynomials
are \(\Pbb\)-almost surely uniquely determined by the 
scaling distributions due to the unisolvency of the points in \(\tau\),
see Corollary~\ref{corollario2}. 

Based on this observation, the convergence of the scaling distributions
at the leaves of the cluster tree directly follows from the convergence
of the corresponding discrete orthonormal polynomials. We have the following
cluster-wise convergence statement, which directly follows from
Proposition~\ref{lem:ort_pol_conv_multi} by rescaling.

\begin{proposition}\label{th:clusterwise_conv}
	Let	\(\widehat\pi_{1}^{\tau,(N)},\ldots,
  \widehat\pi_{|\Lambda_k|}^{\tau,(N)}\) be the
	discrete orthonormal polynomials associated to
	\(\widehat{\Pbb}_{N}^{\tau}\) and let
	\(\widehat\pi_{1}^{\tau},\ldots,\widehat\pi_{|\Lambda_k|}^{\tau}\) be the
	orthonormal polynomials associated to \(\Pbb|_{D_\tau}\), both constructed
	with respect to graded lexicographically ordered monomial basis. Then, 
  for each fixed index \(i\), there holds
	\[
	\widehat{\pi}_{i}^{\tau,(N)}
	\xrightarrow{N\to\infty}
	\widehat{\pi}_{i}^{\tau}
	\]
	uniformly on \(D_\tau\).
\end{proposition}

In particular, by Theorem~\ref{Varadarajan} and
Proposition~\ref{th:clusterwise_conv}
we obtain, for each fixed $\tau,i$,
\begin{equation}\label{eq:cluster-limit}
  \frac{1}{N} F_N \widehat{\pi}_{i}^{\tau,(N)}
  \xrightharpoonup{N\to\infty} \varphi_{i}^{\tau}
  \quad\mathbb{P}\text{-almost surely},
\end{equation}
where $\varphi_{i}^{\tau}$ is the finite signed measure defined by
\begin{equation}\label{eq:varphi-def}
  \varphi_{i}^{\tau}(f)
  \isdef
  \int_D f\widehat{\pi}_{i}^{\tau}\d\mathbb{P},
  \quad f \in\Ccal(D).
\end{equation}
Equivalently, we write $\varphi_{\tau,i} =
\widehat{\pi}_{\tau,i}\mathbb{P}$.
This convergence justifies that we focus on
orthonormal polynomials taking the role of scaling functions
in what follows.

\subsection{Coarsening of broken orthonormal polynomials}
\label{subsec:combine_children}

We next show that the local coarsening step in the refinement relation
\eqref{eq:samplets_refinement} of the samplet construction
recombines the orthonormal polynomials of child clusters to orthonormal
polynomials in the parent cluster. Let \(j<J\) and let
\(\tau\in\Tcal_j\) have children \(\tau_1,\tau_2\in\Tcal_{j+1}\), such that
\(
D_\tau=D_{\tau_1}\cup D_{\tau_2}.
\)
We formulate orthogonality and orthonormality statements 
with respect to the restricted measure \(\Pbb|_{D_\tau}\) on
\(
\Pcal_k^\tau=\{p\mathbbm{1}_{D_\tau}:p\in\Pcal_k\}.
\)

Let
\[
\big\{\widehat\pi_i^{\tau_1}\big\}_{i=1}^{|\Lambda_k|}\subset \Pcal_k^{\tau_1}
\quad\text{and}\quad
\big\{\widehat\pi_i^{\tau_2}\big\}_{i=1}^{|\Lambda_k|}\subset \Pcal_k^{\tau_2}
\]
be the orthonormal polynomials in \(\Pcal_k^{\tau_1}\) and \(\Pcal_k^{\tau_2}\), 
respectively. We consider the space of piecewise polynomials
\[
\Pcal_{k,\mathrm{pw}}^{\tau_1\cup\tau_2}
\isdef\{p_1+p_2: p_1\in \Pcal_k^{\tau_1}, p_2\in \Pcal_k^{\tau_1}\}
\supset\Pcal_k^\tau.
\]
Obviously,
\[
\{\widehat{\pi}^{\tau_1\cup\tau_2}_i\}_{i=1}^{2|\Lambda_k|}
=\{\widehat\pi_i^{\tau_1}\}_{i=1}^{|\Lambda_k|}
\cup\{\widehat\pi_i^{\tau_2}\}_{i=1}^{|\Lambda_k|}
\]
forms an \(L^2(D,\Pbb)\)-orthonormal basis for 
\(\Pcal_{k,\mathrm{pw}}^{\tau_1\cup\tau_2}\).
Herein, we use the graded lexicographic order on each child cluster
and set all basis elements of the first child cluster before
those of the second child cluster. The
moment matrix of \(\tau\) is then given by
\begin{equation}\label{eq:moment_matrix_cont}
	\bs M_j^\tau
	\isdef
	\big[
	(\bs x^{\bs\alpha_i},\widehat{\pi}^{\tau_1\cup\tau_2}_\ell)_{\Pbb|_{D_\tau}}
	\big]_{\substack{i=1,\ldots,|\Lambda_k|\\ \ell=1,\ldots,2|\Lambda_k|}}
	\in\Rbb^{|\Lambda_k|\times 2|\Lambda_k|}.
\end{equation}

\begin{theorem}\label{th:orth_basis}
	Let \((\bs M_j^\tau)^\intercal=\bs Q_j^\tau\bs R_j^\tau\) be the QR decomposition
	of the transposed moment matrix \((\bs M_j^\tau)^\intercal\)
	with \(\bs Q_j^\tau\in\Rbb^{2|\Lambda_k|\times 2|\Lambda_k|}\), 
	\(\bs R_j^\tau\in\Rbb^{2|\Lambda_k|\times |\Lambda_k|}\). If we partition
	\(
	\bs Q_j^\tau=[\bs Q_{j,\bs\Phi}^\tau,\bs Q_{j,\bs\Sigma}^\tau],
	\)
	where
	\(
	\bs Q_{j,\bs\Phi}^\tau\in\Rbb^{2|\Lambda_k|\times |\Lambda_k|},
	\
	\bs Q_{j,\bs\Sigma}^\tau\in\Rbb^{2|\Lambda_k|\times |\Lambda_k|},
	\)
	then the family
	\begin{equation}\label{eq:def_phi_tau}
		\widetilde{\bs\Phi}_j^\tau
		\isdef
		\{\widehat{\pi}^\tau_i\}_{i=1}^{|\Lambda_k|},
		\quad
		\widehat{\pi}^\tau_i
		\isdef
		\sum_{\ell=1}^{2|\Lambda_k|}
		(\bs Q_{j,\bs\Phi}^\tau)_{\ell,i}\widehat{\pi}^{\tau_1\cup\tau_2}_\ell,
	\end{equation}
	is an orthonormal basis of \(\Pcal_k^\tau\), while the family
	\begin{equation}\label{eq:def_sigma_tau}
		\widetilde{\bs\Sigma}_j^\tau
		\isdef
		\{\widetilde{\sigma}_{i}^\tau\}_{i=1}^{|\Lambda_k|},
		\quad
		\widetilde{\sigma}_{i}^\tau
		\isdef
		\sum_{\ell=1}^{2|\Lambda_k|}
		(\bs Q_{j,\bs\Sigma}^\tau)_{\ell,i}\widehat{\pi}^{\tau_1\cup\tau_2}_\ell,
	\end{equation}
	is orthonormal in \(L^2(D_\tau,\Pbb|_{D_\tau})\) and satisfies the vanishing
	moment condition
	\[
	(\bs x^{\bs\alpha_i},\widetilde{\sigma}_{j}^\tau)_{\Pbb|_{D_\tau}}=0,
	\quad
	i,j=1,\ldots,|\Lambda_k|.
	\]
\end{theorem}

\begin{proof}
	Consider the analysis operator
	\[
	T^\tau\colon\Pcal_{k,\mathrm{pw}}^{\tau_1\cup\tau_2}\to\Rbb^{|\Lambda_k|},
	\quad
	T^\tau p
	\isdef
	[(\bs x^{\bs\alpha_i},p)_{\Pbb|_{D_\tau}}]_{i=1}^{|\Lambda_k|}.
	\]
	Identifying \(p\in\Pcal_{k,\mathrm{pw}}^{\tau_1\cup\tau_2}\) 
	with its coordinate vector with respect to the
	orthonormal basis \(\widehat{\pi}^{\tau_1\cup\tau_2}_i\), 
	\(i=1,\ldots,2|\Lambda_k|\) the operator
	\(T^\tau\) is represented by the moment matrix \(\bs M_j^\tau\) from
	\eqref{eq:moment_matrix_cont}.
	
	Since \(\Pcal_{k,\mathrm{pw}}^{\tau_1\cup\tau_2}\) 
	is finite-dimensional, there holds
	\[
	\Pcal_{k,\mathrm{pw}}^{\tau_1\cup\tau_2}
	=\ker(T^\tau)\overset{\perp}{\oplus}\ker(T^\tau)^\perp,
	\quad
	\ker(T^\tau)^\perp=\operatorname{range}\big((T^\tau)^\ast\big).
	\]
	Let \((\bs M_j^\tau)^\intercal=\bs Q_j^\tau\bs R_j^\tau\) be the
	QR decomposition with
	\(
	\bs Q_j^\tau=[\bs Q_{j,\bs\Phi}^\tau,\bs Q_{j,\bs\Sigma}^\tau].
	\)
	Then the columns of \(\bs Q_{j,\bs\Phi}^\tau\) form an orthonormal basis of
	\(\operatorname{range}\big(({\bs M}_j^\tau)^\intercal\big)
  =\ker(\bs M_j^\tau)^\perp\), while the
	columns of \(\bs Q_{j,\bs\Sigma}^\tau\) form an orthonormal basis of
	\(\ker(\bs M_\tau)\). Since \(\bs M_j^\tau\) is the matrix representation of
	\(T^\tau\), this gives
	\[
	\operatorname{span}\widetilde{\bs\Sigma}_j^\tau=\ker(T^\tau).
	\]
	Hence every \(\widetilde{\sigma}\in\widetilde{\bs\Sigma}_j^\tau\) satisfies
	\[
	(\bs x^{\bs\alpha_i},\widetilde{\sigma})_{\Pbb|_{D_\tau}}=0,
	\quad
	i=1,\ldots,|\Lambda_k|,
	\]
	which proves the vanishing moment condition. Since
	\(\widehat{\pi}^{\tau_1\cup\tau_2}_i\), 
	\(i=1,\ldots,2|\Lambda_k|\), is an orthonormal basis and
	\(\bs Q_j^\tau\) is orthogonal, both families
	\(\widetilde{\bs\Phi}_j^\tau\) and \(\widetilde{\bs\Sigma}_j^\tau\) are orthonormal.
	
	Since \(\Pcal_k^\tau\subset\Pcal_{k,\mathrm{pw}}^{\tau_1\cup\tau_2}\)
	and \(\dim\big(\Pcal_{k,\mathrm{pw}}^{\tau_1\cup\tau_2}\big)=2|\Lambda_k|\),
	by a dimension argument, there necessarily holds \(\ker(T^\tau)^\perp=\Pcal_k^\tau\).
	Furthermore,
	\(\widetilde{\bs\Phi}_j^\tau\) is an orthonormal basis in \(\Pcal_k^\tau\).
	From the QR decomposition, we finally infer that
	\(\widehat{\pi}^\tau_i\in
	\operatorname{span}\{{\bs x}^{\bs\alpha_1},\ldots,{\bs x}^{\bs\alpha_i}\}\).
\end{proof}

\begin{remark}
	Under the ordering convention adopted above, the matrix
	\(\bs Q_{j,\bs\Phi}^\tau\) inherits a block upper-triangular structure,
	reflecting the triangular structure of the underlying Gram-Schmidt process.
	For \(|\Lambda_k|=4\), this takes the form
	\[
	\bs Q_{j,\bs\Phi}^\tau =
	\begin{bmatrix}
		r^{\tau_1}_{1,1} & r^{\tau_1}_{1,2} & r^{\tau_1}_{1,3} & r^{\tau_1}_{1,4}\\
		0 & r^{\tau_1}_{2,2} & r^{\tau_1}_{2,3} & r^{\tau_1}_{2,4}\\
		0 & 0 & r^{\tau_1}_{3,3} & r^{\tau_1}_{3,4}\\
		0 & 0 & 0 & r^{\tau_1}_{4,4}\\
		r^{\tau_2}_{1,1} & r^{\tau_2}_{1,2} & r^{\tau_2}_{1,3} & r^{\tau_2}_{1,4}\\
		0 & r^{\tau_2}_{2,2} & r^{\tau_2}_{2,3} & r^{\tau_2}_{2,4}\\
		0 & 0 & r^{\tau_2}_{3,3} & r^{\tau_2}_{3,4}\\
		0 & 0 & 0 & r^{\tau_2}_{4,4}
	\end{bmatrix}.
	\]
\end{remark}

Theorem~\ref{th:orth_basis} shows that the local coarsening step in the
refinement relation~\eqref{eq:samplets_refinement} of the
samplet construction admits a continuous counterpart in terms of orthonormal
polynomials.
Rather than recomputing the orthonormal polynomial family on the parent cell
\(D_\tau\) from scratch, one may recover it from the two child families by
applying a QR decomposition to the associated moment matrix. In this sense,
the same hierarchical mechanism underlying the discrete refinement relation
\eqref{eq:samplets_refinement} persists at the functional level.

\subsection{Extension to downward index sets}
\label{subsec:downward_closed}
So far, our discussion has been formulated in terms of the total-degree
polynomial spaces \(\Pcal_k\), that is, the span of monomials
\(\bs x^{\bs\alpha}\) with \(|\bs\alpha| \le k\). From the perspective of
samplets, this corresponds to imposing the same number of vanishing moments
along every direction. In other words, the cancellation order is
isotropic across all variables. In many applications, however, it is more
natural to replace total-degree truncations by anisotropic index sets. This
situation arises, for instance, when different directions exhibit different
regularity properties, or when one seeks sparse polynomial approximations in
high dimensions, see, for example, \cite{CCS14}. Equivalently, one
prescribes a distinct number of vanishing moments in each coordinate direction.
A standard class of such index sets are the
\emph{downward closed} ones, see, for example, \cite{GG03}. 

\begin{definition}\label{def:downward_closed}
Let \(\Lambda\subset\Nbb_0^d\) be a finite set of multi-indices. We say that
\(\Lambda\) is \emph{downward closed} if for every \(\bs\alpha\in\Lambda\) and
every \(\bs\beta\in\Nbb_0^d\) such that \(\bs\beta\le\bs\alpha\)
component-wise, there holds \(\bs\beta\in\Lambda\).
	
Associated to \(\Lambda\), we define the polynomial space
	\(
	\Pcal_\Lambda\isdef\spn\{\bs x^{\bs\alpha}:\bs\alpha\in\Lambda\}.
	\)
\end{definition}

\begin{remark}\label{rem:td_tp_subcases}
	The total-degree space \(\Pcal_k\) corresponds to the index set
	\[
	\Lambda_k\isdef\{\bs\alpha\in\Nbb_0^d:\ |\bs\alpha|\le k\},
	\]
	while the tensor-product space of degree at most \(k\)	
  corresponds to the index set
	\[
	\Lambda^{\mathrm{tp}}_{k}\isdef
	\{\bs\alpha\in\Nbb_0^d:\alpha_i\le k\ \text{for all }i=1,\ldots,d\}.
	\]
	Both sets are downward closed. More generally, weighted total-degree sets
	and
	hyperbolic crosses fall into the same class. In particular,
  if \(\Lambda\) is downward closed, there exists an enumeration
	\(
	\Lambda=\{{\bs\alpha}_1,\ldots,{\bs\alpha}_{|\Lambda|}\}
	\)
	such that every prefix
	\(
	\Lambda_r\isdef\{{\bs\alpha}_1,\ldots,{\bs\alpha}_r\},
	\quad r=1,\ldots,|\Lambda|,
	\)
	is again downward closed.
\end{remark}

The constructions of
Subsections~\ref{subsec:gram_moment},~\ref{subsec:qr_moments}, and
\ref{subsec:combine_children} extend to downward closed index sets
and the corresponding polynomials spaces \(\Pcal_\Lambda\) with only 
notational modifications. More precisely, one simply replaces the
ordered monomial basis
\(
\{\bs x^{{\bs\alpha}_i}\}_{i=1}^{|\Lambda_k|}
\)
by
\(
\{\bs x^{{\bs\alpha}_i}\}_{i=1}^{|\Lambda|},
\)
and the total-degree space \(\Pcal_k\) by the space \(\Pcal_\Lambda\).
The corresponding Gram matrix is then given by
\[
{\bs G}_\Lambda
\isdef
\big[
({\bs x}^{{\bs\alpha}_i},{\bs x}^{{\bs\alpha}_j})_\Pbb
\big]_{i,j=1}^{|\Lambda|}
\in\Rbb^{|\Lambda|\times |\Lambda|},
\]
while, for a realization \(X_N\), the empirical Gram matrix becomes
\[
{\bs G}_\Lambda^{(N)}
\isdef
\big[
({\bs x}^{{\bs\alpha}_i},{\bs x}^{{\bs\alpha}_j})_{\widehat{\Pbb}_N}
\big]_{i,j=1}^{|\Lambda|}
\in\Rbb^{|\Lambda|\times |\Lambda|}.
\]
Under the same assumptions as before, \({\bs G}_\Lambda\) is positive definite,
and \({\bs G}_\Lambda^{(N)}\) is invertible \(\Pbb\)-almost surely whenever
\(N\ge |\Lambda|\). Moreover, by the same argument as in
Lemma~\ref{lem:moments_conv}, there holds
\[
{\bs G}_\Lambda^{(N)}\xrightarrow{N\to\infty} {\bs G}_\Lambda
\quad\text{entry-wise \(\Pbb\)-almost surely}.
\]
Consequently, the discrete orthogonal polynomials associated with
\(\widehat{\Pbb}_N\), obtained from the ordered monomials
\(\{\bs x^{{\bs\alpha}_i}\}_{i=1}^{|\Lambda|}\), converge uniformly on \(D\)
to their continuous counterparts exactly as in
Proposition~\ref{lem:ort_pol_conv_multi}.

The same applies to the cluster-wise constructions in
Section~\ref{sec:orth_pol_limits}. For a fixed cluster \(\tau\in\Tcal\), we now
consider the restricted polynomial space
\[
\Pcal_\Lambda^\tau
\isdef
\{p\mathbbm{1}_{D_\tau}:p\in\Pcal_\Lambda\},
\]
in place of \(\Pcal_k^\tau\). Then, the corresponding non-degeneracy statement,
the samplet basis construction based QR decomposition of moment matrices
and the convergence of the associated discrete orthonormal polynomials remain
valid after replacing \(\Pcal_k\), \(\Pcal_k^\tau\), and \(|\Lambda_k|\) by
\(\Pcal_\Lambda\), \(\Pcal_\Lambda^\tau\), and \(|\Lambda|\), respectively.

Likewise, the coarsening procedure from
Subsection~\ref{subsec:combine_children} extends without modification.
Indeed, if \(D_\tau=D_{\tau_1}\cup D_{\tau_2}\) is the disjoint union of two
children, one considers the piecewise polynomial space
\[
\Pcal_{\Lambda,\mathrm{pw}}^{\tau_1\cup\tau_2}
\isdef
\Big\{
p_1+p_2
:\ p_1\in\Pcal_{\Lambda}^{\tau_1},p_2\in\Pcal_{\Lambda}^{\tau_1}\Big\},
\]
which is of dimension \(2|\Lambda|\). If
\(
\big\{\widehat\pi_i^{\tau_1}\big\}_{i=1}^{|\Lambda|}
\text{ and }
\big\{\widehat\pi_i^{\tau_2}\big\}_{i=1}^{|\Lambda|}
\)
denote the orthonormal polynomials associated with
\(\Pbb|_{D_{\tau_1}}\) and \(\Pbb|_{D_{\tau_2}}\), then the moment matrix
takes the form
\[
\bs M_\tau^\Lambda
\isdef
\Big[
(\bs x^{{\bs\alpha}_i},
\widehat{\pi}^{\tau_1\cup\tau_2}_i)_{\Pbb|_{D_\tau}}
\Big]_{\substack{\ell=i,\ldots,|\Lambda|\\ \ell=1,\ldots,2|\Lambda|}}
\in\Rbb^{|\Lambda|\times 2|\Lambda|},
\]
where \(\{\widehat{\pi}^{\tau_1\cup\tau_2}_i\}_{i=1}^{2|\Lambda|}\) is the
concatenated orthonormal basis of the child clusters, as before. 
Applying a QR decomposition to
\((\bs M_\tau^\Lambda)^\intercal\) yields, exactly as in
Theorem~\ref{th:orth_basis}, an orthonormal basis of
\(\Pcal_\Lambda^\tau\) together with an orthonormal complement characterized by
the vanishing moment conditions
\[
(\bs x^{{\bs\alpha}_i},\widetilde{\sigma})_{\Pbb|_{D_\tau}}=0,
\quad i=1,\ldots,|\Lambda|.
\]

Therefore, the entire samplet construction and its asymptotic analysis extend
from the total-degree spaces \(\Pcal_k\) to the more general
polynomial spaces \(\Pcal_\Lambda\) associated with downward closed
index sets.
Furthermore, we stress that more general primitives than polynomials
are possible and refer the reader to~\cite{balazs2024construction} for details.

\subsection{Completeness in
\texorpdfstring{$L^2(D,\Pbb)$ as $J\to\infty$}
{L2(D,mu) as J->infinity}}
\label{subsec:L2_completeness}
In the regime where infinitely many refinement levels are considered, the union 
of the local spaces of piecewise polynomials \(\Pcal_{k,\mathrm{pw}}\), over all
clusters and levels, is dense in \(L^2(D,\Pbb)\).
For a precise statement, we first introduce some notation.

Let \(\{\Tcal_n\}_{n\ge0}\) be a nested sequence of binary trees for \(D\), 
with the assumption that, if
\begin{equation}\label{defhn}
h_n=\max_{\tau\in\Lcal(\Tcal_n)}\mathrm{diam}(D_\tau),
\end{equation}
then \(h_n\xrightarrow{n\to\infty}0\).
Define the spaces
\begin{equation}\label{eq:VN_BPS}
V_n \isdef \bigoplus_{\tau\in\Lcal(\Tcal_n)} \Pcal_k^\tau\subset L^2(D,\Pbb),
\end{equation}
clearly \(V_n\subset V_{n+1}\). Moreover, for each \(\tau\in\Lcal(\Tcal_n)\), 
let \(\tau_1,\tau_2\in\Lcal(\Tcal_{n+1})\) denote its two children at the 
next refinement level. Then, the refinement relation, see 
Theorem~\ref{th:orth_basis}, applied on \(D_\tau=D_{\tau_1}\cup D_{\tau_2}\),
yields the orthogonal decomposition
\begin{equation}\label{eq:SIA}
\Pcal_{k,\mathrm{pw}}^{\tau_1\cup\tau_2} 
= \Pcal_k^\tau\ \overset{\perp}{\oplus}\widetilde{\Scal}^\tau.
\end{equation}
A straightforward induction argument using~\eqref{eq:SIA} shows that
\begin{equation}\label{eq:nestedVn}
V_{n+1} = V_n\ \overset{\perp}{\oplus} \widetilde{\Scal}_n,
\quad
\widetilde{\Scal}_n\isdef \bigoplus_{\tau\in \Lcal(\Tcal_n)} 
\widetilde{\Scal}^\tau.
\end{equation}
Let \(\widetilde{\bs\Phi}_0\) be an orthonormal basis of \(V_0\), and for each
\(n\ge0\) let \(\widetilde{\bs\Sigma}_n\) be an orthonormal basis of 
\(\widetilde{\Scal}_n\) obtained by concatenating the local orthonormal bases 
\(\widetilde{\bs\Sigma}_n^\tau\) of \(\widetilde{\Scal}^\tau\) at level \(n\).
Then \(\widetilde{\bs\Phi}_0\cup\bigcup_{n\ge0}\widetilde{\bs\Sigma}_n\) is 
orthonormal in \(L^2(D,\Pbb)\) by the decomposition above. We are now ready to 
prove its completeness.

\begin{theorem}\label{th:L2_basis_clean}
Under the notation above, there holds
\begin{equation}\label{eq:L2_basis_clean}
\overline{\bigcup_{n\ge0}V_n}^{\ \|\cdot\|_{L^2(D,\Pbb)}} = L^2(D,\Pbb).
\end{equation}
Consequently, the orthonormal family 
\(\widetilde{\bs\Phi}_0\cup\bigcup_{n\ge0}\widetilde{\bs\Sigma}_n\) 
obtained from binary refinements is an orthonormal basis of \(L^2(D,\Pbb)\).
\end{theorem}
\begin{proof}
Observe that each space \(V_n\) contains simple functions that are constant 
on each cell \(D_\tau\), \(\tau\in\Lcal(\Tcal_n)\). Hence,
Theorem~\ref{th:L2_basis_clean} follows by the well-known density of
simple functions in \(L^2(D,\Pbb)\). For the sake of self-containment we 
report the details of the approximant's construction.

Let \(f\in \Ccal(D)\). Since \(D\) is compact, \(f\) is uniformly continuous.
For each \(n\in\Nbb\), define \(f_n\) by choosing on every leaf 
\(\tau\in\Lcal(\Tcal_n)\) a point \(\bs x_\tau\in\tau\) and setting
\[
f_n(\bs x)\isdef f(\bs x_\tau),\quad \bs x\in\tau.
\]
Then, \(f_n\) is piecewise constant on \(\Lcal(\Tcal_n)\) and
we have
\[
	\sup_{\bs x\in D_\tau}|f(\bs x)-f(\bs x_\tau)|
  \leq \|f\|_\infty \|\bs x-\bs x_\tau\|_2\leq \|f\|_\infty h_n, 
  \quad \bs x\in D_\tau,
\]
where $h_n$ is defined as in~\eqref{defhn}. Hence, we arrive at
\[
\|f_n-f\|_{L^2(D,\Pbb)}\leq Ch_n \xrightarrow{n\to\infty}0.
\]
Since \(\Ccal(D)\) is dense in \(L^2(D,\Pbb)\),
the density in~\eqref{eq:L2_basis_clean} follows.
The final statement follows from the orthogonal
decompositions~\eqref{eq:nestedVn}.
\end{proof}

\begin{remark}
A systematic study of approximation spaces associated to the constructed
basis is beyond the scope of this article. Nevertheless, let us mention 
that the orthogonal decompositions developed in~\eqref{eq:VN_BPS} suggest 
quantitative approximation estimates for the associated broken 
polynomial spaces, since the spaces \(V_n\) clearly satisfy the same type of
broken Sobolev approximation estimates as discontinuous polynomial 
spaces in finite element theory.
More precisely, let \(k\in \Nbb\), \(1\le p < \infty\) and
define the broken Sobolev norm
\[
  \|v\|_{W^{k,p}_{\mathrm{pw}}(\Tcal_n)}^p\isdef 
\sum_{\tau\in\Lcal(\Tcal_n)} \|v\|_{W^{k,p}(D_\tau)}^p.
\]
Then, the Bramble--Hilbert lemma, see~\cite{BH70}, yields 
\[
  \inf_{v\in V_n} \|f-v\|_{W^{m,p}_{\mathrm{pw}}(\Tcal_n)}
\le
C h_n^{k-m}|f|_{W^{k,p}(D)},
\]
compare, for example, \cite{DL04}.
\end{remark}

\section{Extension of multiwavelets}
\label{sec:alpert_parity_complete}
A widely known construction of polynomial multiwavelets in the univariate setting
is due to Alpert, see~\cite{A93}. On a dyadic partition of an interval, 
compactly supported piecewise polynomial wavelets with vanishing
moments up to a prescribed degree are constructed. The multivariate setting
is addressed by combining the univariate bases in a 
tensor-product construction.
In contrast, the framework developed in Section~\ref{sec:orth_pol_limits} leads
to a genuinely non-tensorial construction based on general polynomial spaces.
In this section, we show that, in the case of congruent binary splits, the
continuous coarsening procedure from Theorem~\ref{th:orth_basis} yields a
detail space with the same vanishing moment structure as in Alpert's
construction. We further show that the corresponding detail functions may be
chosen to satisfy a symmetry condition with respect to the splitting hyperplane.

In what follows, we assume \(D=[0,1]^d\). Then, the
$d$-dimensional Lebesgue measure \(\lambda_d\) is a probability measure and we 
assume \(\Pbb=\lambda_d\), which 
amounts to uniform sampling. 
Let \(j<J\) and let again \(\tau\in\Tcal_j\) have children
\(\tau_1,\tau_2\in\Tcal_{j+1}\), so that
\(
D_\tau=D_{\tau_1}\cup D_{\tau_2}.
\)
Recall the broken polynomial
space
\[
\Pcal_{k,\mathrm{pw}}^{\tau_1\cup\tau_2}
=\{p_1+p_2: p_1\in \Pcal_k^{\tau_1}, p_2\in \Pcal_k^{\tau_1}\}
\supset\Pcal_k^\tau,
\]
as well as the orthogonal decomposition
\[
\Pcal_{k,\mathrm{pw}}^{\tau_1\cup\tau_2}
=
\Pcal_k^\tau \overset{\perp}{\oplus}\widetilde\Scal_j^\tau,
\]
from Subsection~\ref{subsec:combine_children}. 
Herein, \(\Pcal_k^\tau\) is defined as in~\eqref{eq:defPkDtau} and
\[
\widetilde\Scal_j^\tau\isdef\operatorname{span}\widetilde{\bs\Sigma}_j^\tau,
\quad
\widetilde{\bs\Sigma}_j^\tau
=\{\widetilde{\sigma}_{j,\ell}^\tau\}_{\ell=1}^{|\Lambda_k|},
\]
with \(\widetilde{\bs\Sigma}_j^\tau\) given by 
Theorem~\ref{th:orth_basis}. In particular,
\(\widetilde{\Scal}_j^\tau\) consists precisely of those functions 
\(\widetilde{\sigma} \in\Pcal_{k,\mathrm{pw}}^{\tau_1\cup\tau_2}\) that
satisfy the vanishing moment conditions
\[
(\bs x^{\bs\alpha_i},\widetilde{\sigma})_{\Pbb}=0,
\quad i=1,\ldots,|\Lambda_k|.
\]

In the univariate construction, the multiwavelets associated with
symmetric dyadic splits may be chosen to be even or odd with respect to the
midpoint of the parent interval. 
Specifically, Alpert worked on the interval \([0,1]\), which is subdivided dyadically and, 
for a level \(j\), any multiwavelet \(f_{k,j}\) is a piecewise-defined 
function satisfying
\[
f_{k,j}(x)=\begin{cases}
	p(x) & \text{if } x\in [k2^{-j},(2k+1)2^{-(j+1)}],\\
	\pm p(-x) & \text{if } x\in [(2k+1)2^{-(j+1)},(k+1)2^{-j}],\\
	0 & \text{otherwise},
\end{cases}
\]
for a suitable polynomial \(p\) and a suitable choice of sign.
We show that an analogous symmetry property
holds in the present setting for symmetric binary splits in arbitrary space
dimension.

Assume that \(D_\tau\) is bisected by the hyperplane orthogonal to the
direction \(\bs e_o\) through its center \(\bs c_\tau\), for some
\(o\in\{1,\ldots,d\}\), so that \(D_{\tau_1}\) and \(D_{\tau_2}\) are
exchanged by the reflection
\[
\rho_\tau(\bs x)
\isdef
\bs x-2\big((\bs x-\bs c_\tau)\cdot \bs e_o\big)\bs e_o,
\quad \bs x\in D_\tau.
\]
Observe that \(\Pbb|_{D_\tau}\) is invariant under
\(\rho_\tau\). We then define the associated reflection operator
\[
\mathcal R_\tau\colon L^2(D_\tau,\Pbb|_{D_\tau})
\to L^2(D_\tau,\Pbb|_{D_\tau}),\quad
(\mathcal R_\tau f)(\bs x)\isdef f\big(\rho_\tau(\bs x)\big).
\]
Observe that \(\mathcal{R}_\tau\) is unitary and involutive, 
so its eigenvalues \(\lambda\) satisfy \(|\lambda|=1\).
Accordingly, we introduce the subspaces
\begin{align}
	&W^+_\tau\isdef \{f\in\Pcal_{k,\mathrm{pw}}^{\tau_1\cup\tau_2}:
  \mathcal R_\tau f=f\},\\
	&W^-_\tau\isdef \{f\in\Pcal_{k,\mathrm{pw}}^{\tau_1\cup\tau_2}:
  \mathcal R_\tau f=-f\},
\end{align}
and obtain the orthogonal decomposition
\begin{equation}\label{eq:W_even_odd}
	\Pcal_{k,\mathrm{pw}}^{\tau_1\cup\tau_2}
  =W^+_\tau\overset{\perp}{\oplus}W^-_\tau.
\end{equation}

Let \(h_o>0\) denote half the side length of \(D_\tau\) in direction
\(\bs e_o\), and define the centered scaled coordinate
\[
t_o(\bs x)\isdef \frac{x_o-(\bs c_\tau)_o}{h_o},
\quad \bs x\in D_\tau.
\]
For \(\bs\beta\in\Lambda_k\), we set
\begin{equation}\label{eq:centered_basis_binary}
	\widehat e_{\bs\beta}(\bs x)
	\isdef
	t_o(\bs x)^{\beta_o}
	\prod_{\substack{i=1\\ i\neq o}}^d x_i^{\beta_i}.
\end{equation}
Since \(t_o(\bs x)\) is affine in \(x_o\), the family
\(\{\widehat e_{\bs\beta}\}_{\bs\beta\in\Lambda_k}\) spans \(\Pcal_k\), and
hence \(\{\widehat e_{\bs\beta}\mathbbm{1}_{D_\tau}\}_{\bs\beta\in\Lambda_k}\)
spans \(\Pcal_k^\tau\). Therefore, for 
\(\widetilde{\sigma}\in\Pcal_{k,\mathrm{pw}}^{\tau_1\cup\tau_2}\), the condition
\(\widetilde{\sigma}\in \widetilde\Scal_j^\tau\) is equivalent to
\begin{equation}\label{eq:centered_moments_binary}
	(\widehat e_{\bs\beta},\widetilde{\sigma})_{\Pbb|_{D_\tau}}=0,
	\quad
	\text{for all }\bs\beta\in\Lambda_k.
\end{equation}
Moreover, by construction of \(\rho_\tau\), there holds
\begin{equation}\label{eq:tilde_parity_binary}
	\widehat e_{\bs\beta}\big(\rho_\tau(\bs x)\big)
	=
	(-1)^{\beta_o}\widehat e_{\bs\beta}(\bs x),
	\quad \bs x\in D_\tau.
\end{equation}

We are now in the position to show that the space \(\widetilde\Scal_j^\tau\) 
admits an orthonormal
basis with definite symmetry.

\begin{theorem}\label{th:multid_parity_basis_binary}
	Let \(\tau\in\Tcal_j\) be such that \(D_\tau\) is symmetrically split as
	above. Then,
	\begin{equation}\label{eq:D_split_binary}
		\widetilde\Scal_j^\tau
		=
		\bigl(\widetilde\Scal_j^\tau\cap W^+_\tau\bigr)
		\overset{\perp}{\oplus}
		\bigl(\widetilde\Scal_j^\tau\cap W^-_\tau\bigr).
	\end{equation}
	In particular, \(\widetilde\Scal_j^\tau\) admits an orthonormal basis
	\(\{\widetilde{\sigma}_{j,\ell}^\tau\}_{\ell=1}^{|\Lambda_k|}\) such that each
	\(\widetilde{\sigma}_{j,\ell}^\tau\) belongs either to \(W^+_\tau\) or to
	\(W^-_\tau\), and hence satisfies
	\[
	\mathcal R_\tau\widetilde{\sigma}_{j,\ell}^\tau=\pm 
  \widetilde{\sigma}_{j,\ell}^\tau,
	\quad \ell=1,\ldots,|\Lambda_k|.
	\]
\end{theorem}

The proof of this theorem requires two preliminary lemmas.

\begin{lemma}\label{lem:parity_decouple_multid}
	If \(\widetilde{\sigma}\in W^+_\tau\), then
	\[
	(\widehat e_{\bs\beta},\widetilde{\sigma})_{\Pbb|_{D_\tau}}=0
	\]
	for every \(\bs\beta\in\Lambda_k\) such that \(\beta_o\) is odd.
	If \(\widetilde{\sigma}\in W^-_\tau\), then
	\[
	(\widehat e_{\bs\beta},\widetilde{\sigma})_{\Pbb|_{D_\tau}}=0
	\]
	for every \(\bs\beta\in\Lambda_k\) such that \(\beta_o\) is even.
\end{lemma}
\begin{proof}
	Assume first that \(\widetilde{\sigma}\in W^+_\tau\). Let 
  \(\bs\beta\in\Lambda_k\) with
	\(\beta_o\) odd. Using the change of variables
	\(\bs x\mapsto \rho_\tau(\bs x)\), the invariance of \(\Pbb|_{D_\tau}\),
	\eqref{eq:tilde_parity_binary}, and the identity
	\(\widetilde\sigma\big(\rho_\tau(\bs x)\big)=\widetilde\sigma(\bs x)\), we obtain
	\begin{align*}
		(\widehat e_{\bs\beta},\widetilde{\sigma})_{\Pbb|_{D_\tau}}
		&=
		\int_{D_\tau}\widehat e_{\bs\beta}(\bs x)\widetilde{\sigma}(\bs x)\d\Pbb
		=
		\int_{D_\tau}\widehat e_{\bs\beta}\big(\rho_\tau(\bs x)\big)
		\widetilde{\sigma}\big(\rho_\tau(\bs x)\big)\d\Pbb\\
		&=
		-\int_{D_\tau}\widehat e_{\bs\beta}(\bs x)\widetilde{\sigma}(\bs x)\d\Pbb
		=-(\widehat e_{\bs\beta},\widetilde{\sigma})_{\Pbb|_{D_\tau}}.
	\end{align*}
	Hence the integral vanishes. The case \(\widetilde{\sigma}\in W^-_\tau\) 
  is analogous.
\end{proof}

\begin{lemma}\label{lem:D_invariant_rewrite}
	The detail space \(\widetilde\Scal_j^\tau\) is invariant 
  under \(\mathcal R_\tau\).
\end{lemma}
\begin{proof}
	Let \(\widetilde{\sigma}\in \widetilde\Scal_j^\tau\). 
  By~\eqref{eq:centered_moments_binary}, it
	suffices to show that
	\[
	(\widehat e_{\bs\beta},\mathcal R_\tau \widetilde{\sigma})_{\Pbb|_{D_\tau}}=0
	\quad\text{for all }\bs\beta\in\Lambda_k.
	\]
	Using again the change of variables \(\bs x\mapsto\rho_\tau(\bs x)\) and
	\eqref{eq:tilde_parity_binary}, we find
	\begin{align*}
		(\widehat e_{\bs\beta},\mathcal R_\tau \widetilde{\sigma})_{\Pbb|_{D_\tau}}
		=
		\int_{D_\tau}\widehat e_{\bs\beta}(\bs x)\widetilde{\sigma}
    \big(\rho_\tau(\bs x)\big)\d\Pbb
		=
		\int_{D_\tau}\widehat e_{\bs\beta}\big(\rho_\tau(\bs x)\big)
    \widetilde{\sigma}(\bs x)\d\Pbb
		=
		(-1)^{\beta_o}
		(\widehat e_{\bs\beta},\widetilde{\sigma})_{\Pbb|_{D_\tau}}
		=0.
	\end{align*}
	This proves the claim.
\end{proof}

We are now ready to prove the theorem.

\begin{proof}[Proof of Theorem~\ref{th:multid_parity_basis_binary}]
	Let \(T^\tau \colon \Pcal_{k,\mathrm{pw}}^{\tau_1\cup\tau_2}
  \to\Rbb^{|\Lambda_k|}\) be the analysis operator
	defined by
	\[
	(T^\tau \widetilde{\sigma})_{\bs\beta}
	\isdef
	(\widehat e_{\bs\beta},\widetilde{\sigma})_{\Pbb|_{D_\tau}},
	\quad \bs\beta\in\Lambda_k.
	\]
	We define
	\[
	\widetilde\Scal_j^{\tau,+}\isdef \widetilde\Scal_j^\tau\cap W^+_\tau,
	\quad
	\widetilde\Scal_j^{\tau,-}\isdef \widetilde\Scal_j^\tau\cap W^-_\tau.
	\]
	By~\eqref{eq:centered_moments_binary}, there holds
	\[
	\widetilde\Scal_j^\tau=\ker(T^\tau).
	\]
	
	First, we prove~\eqref{eq:D_split_binary}. The inclusion \((\supseteq)\) is
	trivial, so it suffices to show \((\subseteq)\). 
  Let \(\widetilde{\sigma}\in\widetilde\Scal_j^\tau\).
	By~\eqref{eq:W_even_odd}, there exist \(\sigma_+\in W^+_\tau\) and
	\(\widetilde{\sigma}_-\in W^-_\tau\) such that
	\[
	\widetilde{\sigma}=\widetilde{\sigma}_++\widetilde{\sigma}_-.
	\]
	It remains to show that \(\widetilde{\sigma}_\pm\in\ker(T^\tau)\), i.e.,
	\(T^\tau \widetilde{\sigma}_\pm=0\). Since \(T^\tau \widetilde{\sigma}=0\), 
  there holds
	\begin{equation}\label{eq:Tf1Tf2}
		T^\tau \widetilde{\sigma}_+=-T^\tau \widetilde{\sigma}_-.
	\end{equation}
	By Lemma~\ref{lem:parity_decouple_multid}, if \(\bs\beta\in\Lambda_k\) is such
	that \(\beta_o\) is odd, then
	\[
	(\widehat e_{\bs\beta},\widetilde{\sigma}_+)_{\Pbb|_{D_\tau}}=0,
	\]
	whereas
	\[
	(\widehat e_{\bs\beta},\widetilde{\sigma}_-)_{\Pbb|_{D_\tau}}=0
	\]
	if \(\beta_o\) is even. Reading~\eqref{eq:Tf1Tf2} component-wise yields
	\[
	(\widehat e_{\bs\beta},\widetilde{\sigma}_+)_{\Pbb|_{D_\tau}}
	=
	-(\widehat e_{\bs\beta},\widetilde{\sigma}_-)_{\Pbb|_{D_\tau}}
	=0
	\]
	also for those \(\bs\beta\in\Lambda_k\) with \(\beta_o\) even.
	Consequently,
	\[
	(\widehat e_{\bs\beta},\widetilde{\sigma}_+)_{\Pbb|_{D_\tau}}=0
	\quad\text{for every }\bs\beta\in\Lambda_k,
	\]
	i.e., \(T^\tau \widetilde{\sigma}_+=0\). By~\eqref{eq:Tf1Tf2}, we also obtain
	\(T^\tau \widetilde{\sigma}_-=0\). Hence
	\[
	\widetilde{\sigma}_+\in\widetilde\Scal_j^\tau\cap W^+_\tau,
	\quad
	\widetilde{\sigma}_-\in\widetilde\Scal_j^\tau\cap W^-_\tau,
	\]
	which proves~\eqref{eq:D_split_binary}.
	
	The second assertion follows by choosing orthonormal bases of
	\(\widetilde\Scal_j^{\tau,+}\) and \(\widetilde\Scal_j^{\tau,-}\) and 
  taking their union.
\end{proof}

While Theorem~\ref{th:multid_parity_basis_binary} guarantees the existence of a
symmetric orthonormal basis of \(\widetilde\Scal_j^\tau\), its proof is 
not constructive
and therefore does not provide such a basis explicitly. To fill this gap,
observe that if \(\{\widetilde{\sigma}_\ell\}_{\ell=1}^{|\Lambda_k|}\) is any
orthonormal basis of \(\widetilde\Scal_j^\tau\), then
\[
\widetilde{\sigma}_\ell^\pm
\isdef
\frac12\bigl(\widetilde{\sigma}_\ell\pm\mathcal R_\tau\widetilde{\sigma}_\ell\bigr)
\]
spans \(\widetilde\Scal_j^\tau\), but does not form a basis of 
\(\widetilde\Scal_j^\tau\), since
\(\dim \widetilde\Scal_j^\tau=|\Lambda_k|<2|\Lambda_k|\). 
After eliminating redundancies,
one may apply the Gram-Schmidt process separately to the families
\(\{\widetilde{\sigma}_\ell^+\}_{\ell=1}^{|\Lambda_k|}\) 
and \(\{\widetilde{\sigma}_\ell^-\}_{\ell=1}^{|\Lambda_k|}\) 
to obtain orthonormal bases of \(\widetilde\Scal_j^\tau\cap W^+_\tau\) and
\(\widetilde\Scal_j^\tau\cap W^-_\tau\), respectively. By taking the union of 
these two bases, one obtains an orthonormal basis of \(\widetilde\Scal_j^\tau\) 
with the desired symmetries.

\section{Numerical results}\label{sec:numerics}
We numerically study the convergence of the samplet basis using random samples
as well as Halton points in \(D=[0,1]^d\) for \(d=1,2,3\). Exemplarily, we
consider samplets with \(k+1=3\) total-degree vanishing moments, i.e.,
the index set is given by \(\Lambda_2\). The maximum level of the cluster
tree is given by \(J=10\) \((d=1)\), \(J=5\) \((d=2)\) and
\(J=3\) \((d=3)\). The minimum number of samples per leaf is chosen such
that \(|\tau|\geq|\Lambda_2|\) for all \(\tau\in\Lcal(\Tcal)\). The specific
average minimum numbers of samples per leave in case of 10 runs of random sampling
and the minimum number of samples in case of Halton points are shown on the left
and on the right of Figure~\ref{fig:numberOfSamples}, respectively.
\begin{figure}[htb]
  \centering
  \begin{tikzpicture}
  \begin{axis}[
    width=0.44\textwidth,
    xlabel={$N$},
    ylabel={$\min_{\tau\in\Lcal(\Tcal)}|\tau|$},
    grid=both,
    legend style={at={(0.1,0.98)},anchor=north west},
    error bars/y dir=both,
    error bars/y explicit,
        xmode=log,
        ymode=log
  ]
    \addplot[
      color=blue,
      mark=*,
    ]
      table[
        header=false,
        x index=0,         
        y index=1,         
      ]{./data/ave_sd_error_1_MC_stats.txt};
    \addlegendentry{$d=1$};

    \addplot[
      color=red,
      mark=square*,
    ]
      table[
        header=false,
        x index=0,
        y index=1,
      ]{./data/ave_sd_error_2_MC_stats.txt};
    \addlegendentry{$d=2$};

    \addplot[
      color=green!60!black,
      mark=triangle*,
    ]
      table[
        header=false,
        x index=0,
        y index=1,
      ]{./data/ave_sd_error_3_MC_stats.txt};
    \addlegendentry{$d=3$};
  \end{axis}
\end{tikzpicture}\qquad
  \begin{tikzpicture}
  \begin{axis}[
    width=0.44\textwidth,
    xlabel={$N$},
    ylabel={$\min_{\tau\in\Lcal(\Tcal)}|\tau|$},
    grid=both,
    legend style={at={(0.1,0.98)},anchor=north west},
    error bars/y dir=both,
    error bars/y explicit,
        xmode=log,
        ymode=log
  ]
    \addplot[
      color=blue,
      mark=*,
    ]
      table[
        header=false,
        x index=0,         
        y index=1,         
      ]{./data/ave_sd_error_1_QMC_stats.txt};
    \addlegendentry{$d=1$};

    \addplot[
      color=red,
      mark=square*,
    ]
      table[
        header=false,
        x index=0,
        y index=1,
      ]{./data/ave_sd_error_2_QMC_stats.txt};
    \addlegendentry{$d=2$};

    \addplot[
      color=green!60!black,
      mark=triangle*,
    ]
      table[
        header=false,
        x index=0,
        y index=1,
      ]{./data/ave_sd_error_3_QMC_stats.txt};
    \addlegendentry{$d=3$};
  \end{axis}
\end{tikzpicture}
\caption{\label{fig:numberOfSamples}Number of samples $N$ and 
minimum number of samples per tree leaf for random sampling 
(left, average over 10 runs)
and Halton points (right).}
\end{figure}
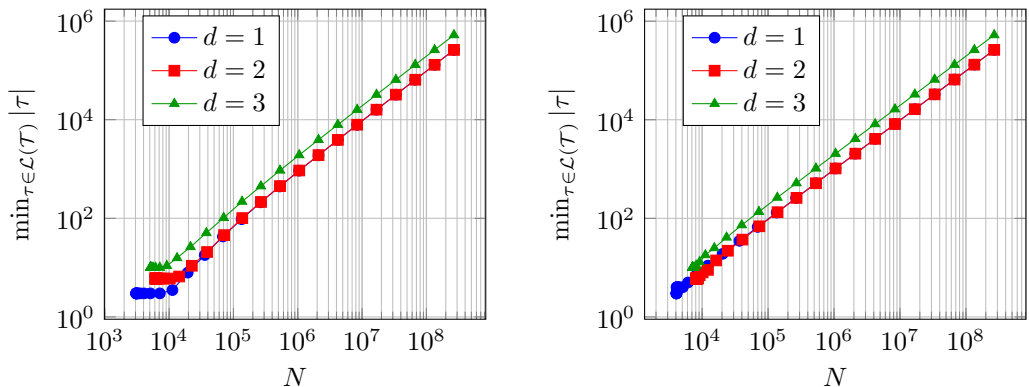

In Figure~\ref{fig:visConv2D} and Figure~\ref{fig:MCcnv2D}, we visualize the
convergence of a scaling distribution (left) and a samplet (right) for 
different numbers of samples
\(N\). As can be seen, already for a relatively small number of samples, 
the coefficients
of the corresponding signed measures are very close to each other.
\begin{figure}[htb]
\centering
\includegraphics[scale=0.3]{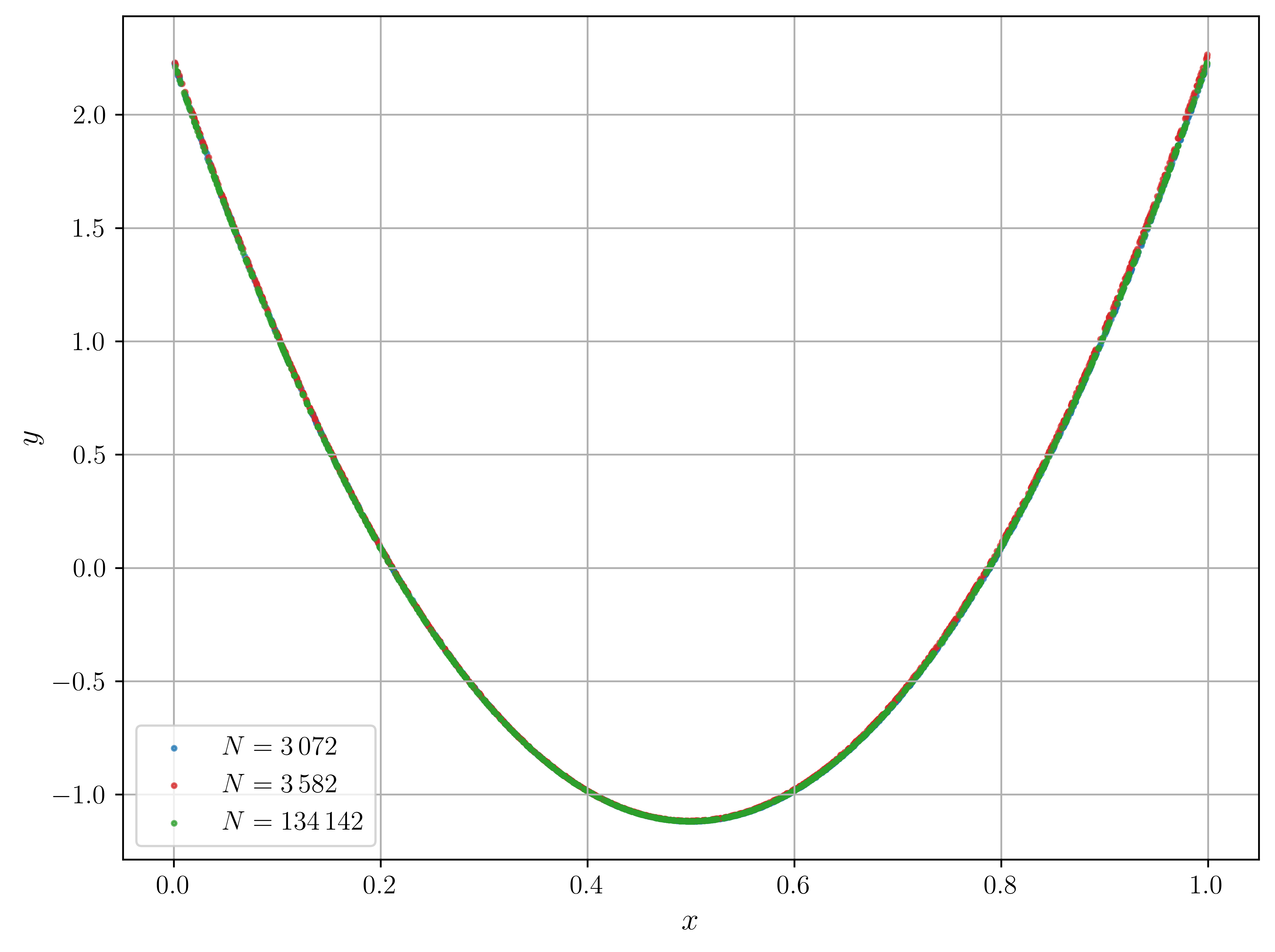}\qquad
\includegraphics[scale=0.3]{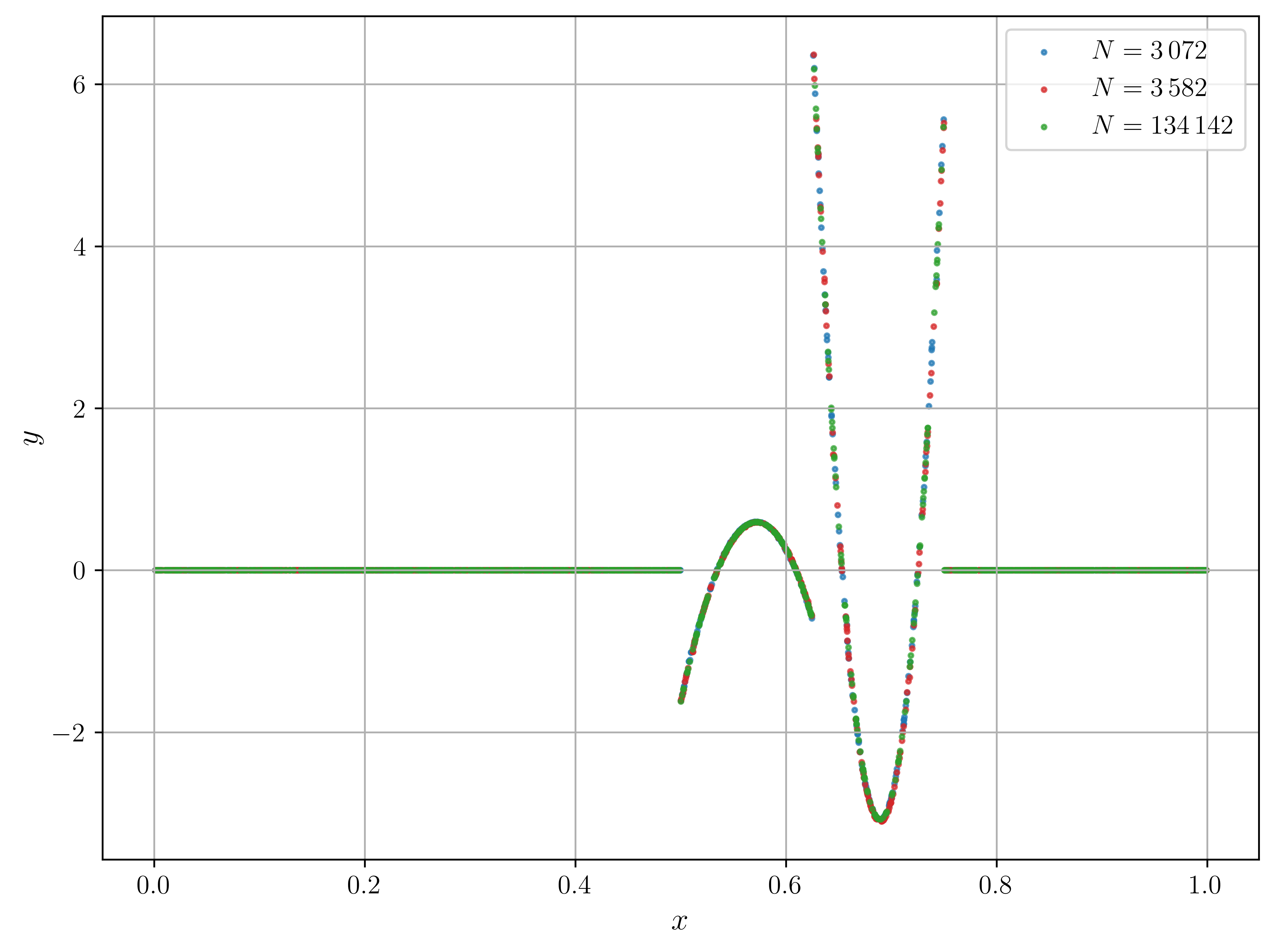}
\caption{\label{fig:visConv1D}Visualization of the samplet convergence in 
  \(d=1\) for increasing
number of samples. A scaling distribution is shown on the left and a samplet 
on the right.}
\end{figure}
\begin{figure}[htb]
\centering
\includegraphics[scale=0.35]{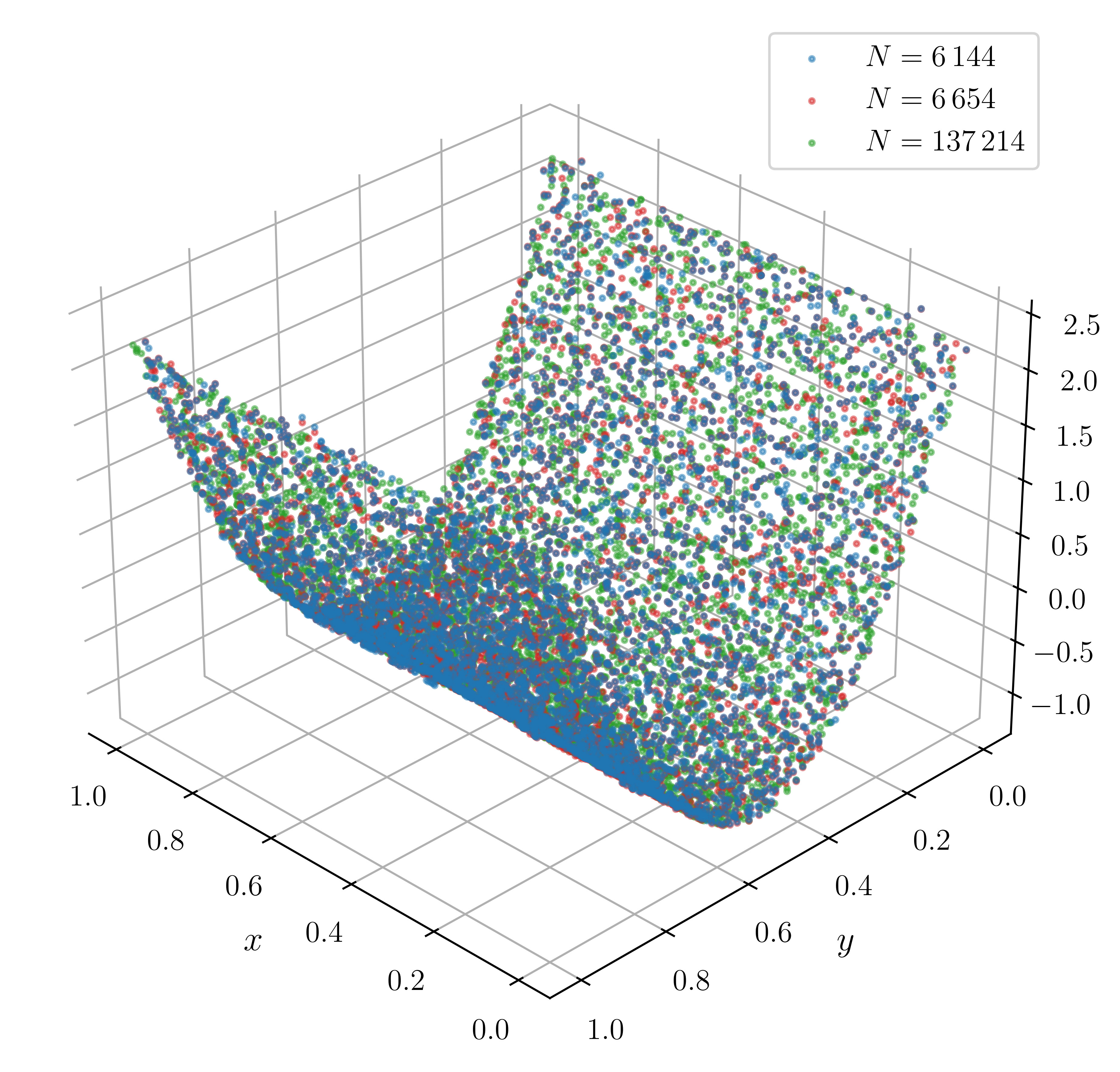}\qquad
\includegraphics[scale=0.35]{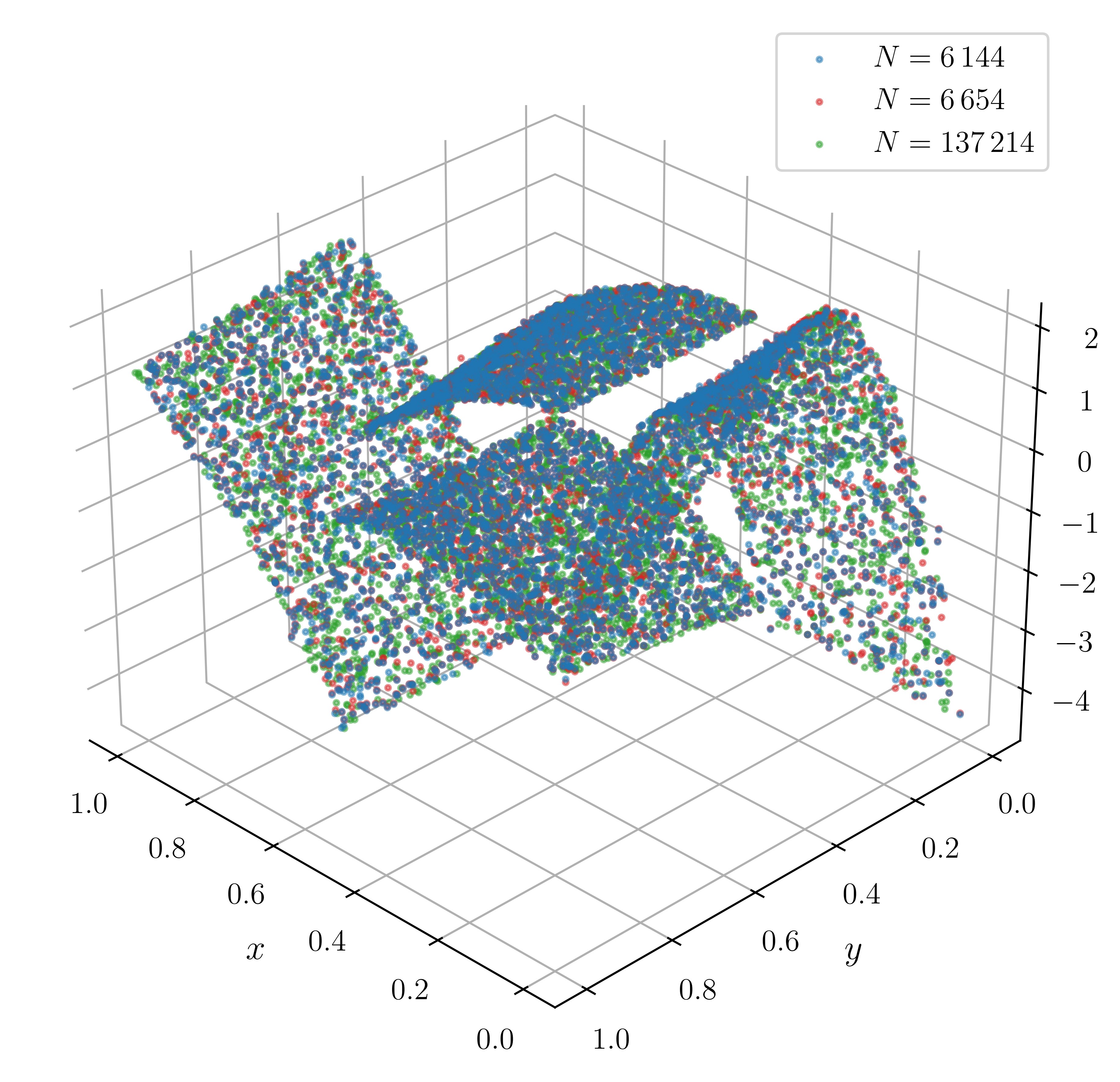}
\caption{\label{fig:visConv2D}Visualization of the samplet convergence in 
  \(d=2\) for increasing
number of samples. A scaling distribution is shown on the left and a samplet 
on the right.}
\end{figure}

To study the convergence, we employ an approximation with a larger number 
of samples as a reference. In all cases, we approximately use 
\(5.37\cdot 10^8\) samples for the computation
of the references. The exact numbers are given in Table~\ref{tab:refSamples}.
This number is required to have sufficiently many samples per leaf to enter 
the asymptotic regime.

\begin{table}[htb]
\centering
\begin{tabular}{l|c|c|c}
 & $d=1$ & $d = 2$ & $d = 3$\\\hline
Random &$536\,873\,982\ (521\,903)$ &$536\,877\,054\ (521\,853)$ 
       & $536\,876\,030\ (1\,045\,541)$\\\hline
Halton & $536\,874\,982\ (524\,291)$ & $536\,879\,054\ (524\,291)$ 
       & $536\,878\,030\ (1\,048\,578)$
\end{tabular}
\caption{\label{tab:refSamples}Number of samples \(N\) used as a reference 
  for error computation. Numbers in parentheses correspond to the (average) 
  minimum leaf sizes for
random points and Halton points.}
\end{table}

To benchmark the convergence, we consider the average
projection error of the filter coefficients of all non-leaf clusters, 
where we distinguish
between scaling distributions and samplets. In each cluster, the projection 
error is computed
as
\[
e^\tau_{\bs\Xi}\isdef\frac{\|{\bs Q}^\tau_{\bs\Xi}
-{\bs Q}^\tau_{{\bs\Xi},\mathrm{ref}}({\bs Q}^\tau_{{\bs\Xi},
\mathrm{ref}})^\intercal{\bs Q}^\tau\|_F}
{\|{\bs Q}^\tau_{{\bs\Xi},\mathrm{ref}}\|_F},
\quad{\bs\Xi}\in\{\bs\Phi,\bs\Sigma\}.
\]
The reported error is then 
\(\frac{1}{|\Tcal\setminus\Lcal(\Tcal)|}
\sum_{\tau\in\Tcal\setminus\Lcal(\Tcal)}e^\tau_{\bs\Xi}\),
\({\bs\Xi}\in\{\bs\Phi,\bs\Sigma\}\). 
On the left of Figure~\ref{fig:MCcnv2D}, we see the convergence
of the filter coefficients of the scaling distributions, while the errors
of the samplets' filter
coefficients are shown on the right. The error is computed by averaging 10 
runs and the error bars indicate one standard deviation.
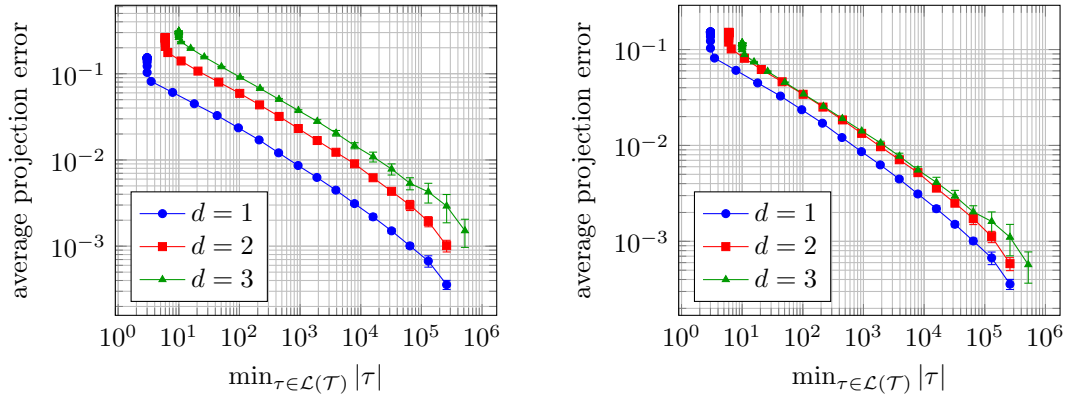
\begin{figure}[htb]
  \centering
  \begin{tikzpicture}
  \begin{axis}[
    width=0.44\textwidth,
    xlabel={$\min_{\tau\in\Lcal(\Tcal)}|\tau|$},
    ylabel={average projection error},
    grid=both,
    legend style={at={(0.04,0.4)},anchor=north west},
    error bars/y dir=both,
    error bars/y explicit,
        xmode=log,
        ymode=log,
        mark size=1.5pt
  ]
    \addplot[
      color=blue,
      mark=*,
    ]
      table[
        header=false,
        x index=1,         
        y index=2,         
        y error index=3,   
      ]{./data/ave_sd_error_1_MC_stats.txt};
    \addlegendentry{$d=1$};

    \addplot[
      color=red,
      mark=square*,
    ]
      table[
        header=false,
        x index=1,
        y index=2,
        y error index=3,
      ]{./data/ave_sd_error_2_MC_stats.txt};
    \addlegendentry{$d=2$};

    \addplot[
      color=green!60!black,
      mark=triangle*,
    ]
      table[
        header=false,
        x index=1,
        y index=2,
        y error index=3,
      ]{./data/ave_sd_error_3_MC_stats.txt};
    \addlegendentry{$d=3$};
  \end{axis}
\end{tikzpicture}\qquad
  \begin{tikzpicture}
  \begin{axis}[
    width=0.44\textwidth,
    xlabel={$\min_{\tau\in\Lcal(\Tcal)}|\tau|$},
    ylabel={average projection error},
    grid=both,
    legend style={at={(0.04,0.4)},anchor=north west},
    error bars/y dir=both,
    error bars/y explicit,
        xmode=log,
        ymode=log,
        mark size=1.5pt
  ]
    \addplot[
      color=blue,
      mark=*,
    ]
      table[
        header=false,
        x index=1,         
        y index=2,         
        y error index=3,   
      ]{./data/ave_smp_error_1_MC_stats.txt};
    \addlegendentry{$d=1$};

    \addplot[
      color=red,
      mark=square*,
    ]
      table[
        header=false,
        x index=1,
        y index=2,
        y error index=3,
      ]{./data/ave_smp_error_2_MC_stats.txt};
    \addlegendentry{$d=2$};

    \addplot[
      color=green!60!black,
      mark=triangle*,
    ]
      table[
        header=false,
        x index=1,
        y index=2,
        y error index=3,
      ]{./data/ave_smp_error_3_MC_stats.txt};
    \addlegendentry{$d=3$};
  \end{axis}
\end{tikzpicture}
\caption{\label{fig:MCcnv2D}Average projection error of filter coefficients
  of scaling distribution (left) and samplets (right) taken over all non-leaf
  clusters for increasing \(N\) and uniformly random points.
  The error bars denote the standard deviation
from 10 runs.}
\end{figure}
The rate of convergence resembles the typical Monte-Carlo rate \(N^{-1/2}\)
for \(d=1,2,3\). The standard deviation between the different runs is relatively
small. 

Figure~\ref{fig:QMCcnv2D} depicts the corresponding errors in case of 
Halton points.
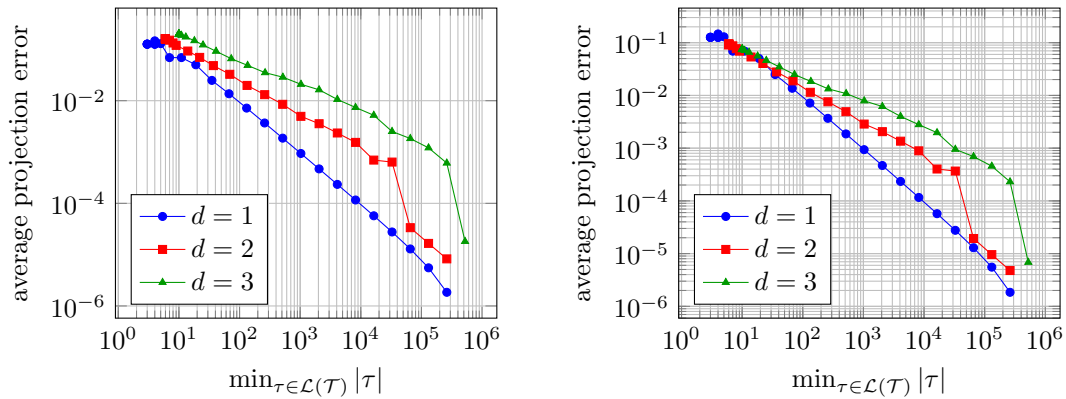
\begin{figure}[htb]
  \centering
  \begin{tikzpicture}
  \begin{axis}[
    width=0.44\textwidth,
    xlabel={$\min_{\tau\in\Lcal(\Tcal)}|\tau|$},
    ylabel={average projection error},
    grid=both,
    legend style={at={(0.04,0.4)},anchor=north west},
    error bars/y dir=both,
    error bars/y explicit,
        xmode=log,
        ymode=log,
        mark size=1.5pt
  ]
    \addplot[
      color=blue,
      mark=*,
    ]
      table[
        header=false,
        x index=1,         
        y index=2,         
      ]{./data/ave_sd_error_1_QMC_stats.txt};
    \addlegendentry{$d=1$};

    \addplot[
      color=red,
      mark=square*,
    ]
      table[
        header=false,
        x index=1,
        y index=2,
      ]{./data/ave_sd_error_2_QMC_stats.txt};
    \addlegendentry{$d=2$};

    \addplot[
      color=green!60!black,
      mark=triangle*,
    ]
      table[
        header=false,
        x index=1,
        y index=2,
      ]{./data/ave_sd_error_3_QMC_stats.txt};
    \addlegendentry{$d=3$};
  \end{axis}
\end{tikzpicture}\qquad
  \begin{tikzpicture}
  \begin{axis}[
    width=0.44\textwidth,
    xlabel={$\min_{\tau\in\Lcal(\Tcal)}|\tau|$},
    ylabel={average projection error},
    grid=both,
    legend style={at={(0.04,0.4)},anchor=north west},
    error bars/y dir=both,
    error bars/y explicit,
        xmode=log,
        ymode=log,
        mark size=1.5pt
  ]
    \addplot[
      color=blue,
      mark=*,
    ]
      table[
        header=false,
        x index=1,         
        y index=2,         
      ]{./data/ave_smp_error_1_QMC_stats.txt};
    \addlegendentry{$d=1$};

    \addplot[
      color=red,
      mark=square*,
    ]
      table[
        header=false,
        x index=1,
        y index=2,
      ]{./data/ave_smp_error_2_QMC_stats.txt};
    \addlegendentry{$d=2$};

    \addplot[
      color=green!60!black,
      mark=triangle*,
    ]
      table[
        header=false,
        x index=1,
        y index=2,
      ]{./data/ave_smp_error_3_QMC_stats.txt};
    \addlegendentry{$d=3$};
  \end{axis}
\end{tikzpicture}
\caption{\label{fig:QMCcnv2D}Average projection error of filter coefficients
  of scaling distribution (left) and samplets (right) taken over all non-leaf
clusters for increasing \(N\) for Halton points}
\end{figure}
As expected, the convergence using quasi-random points is faster than
for random points. However, for smaller numbers of samples per leaf, we 
observe a rate that is significantly worse then the expected rate of 
\(N^{-1+\varepsilon}\), \(0<\varepsilon\ll 1\), for \(d=2,3\). The error then
suddenly drops for larger number of samples, which suggests that
some clusters are in the beginning poorly resolved by the sample points.

\section{Conclusions}\label{sec:conclusion}
In a probabilistic framework, we have developed a continuous limit theory
for samplets and characterized the deterministic multiresolution structure
that arises as sampling becomes dense. Specifically, we have proven the
uniform convergence of discrete orthogonal polynomials to their continuous
counterparts, as well as the compatibility of this limiting procedure with
the recursive definition of samplets. Building on these results, we have
established the convergence of the samplet basis to a continuous multivariate
framework of compactly supported signed measures with broken polynomial
densities, which constitute polynomial multiwavelets in the infinite data 
limit. We have also discussed how the theory extends from total-degree spaces
to more general downward closed index sets, thereby accommodating anisotropic
moment conditions aligned with sparse and high-dimensional approximation goals.
In the case of symmetric binary splits of the unit hypercube, the construction
recovers Alpert-type multiwavelets, including their symmetry and scale- and
partition- independent filter coefficients, without resorting to tensor-product
constructions. Finally, we have studied the convergence of samplets numerically,
for both random and low-discrepancy data sites. The illustrative experiments 
quantitatively corroborate the theoretical convergence results.

\section*{Acknowledgment}
The authors have been funded by the Swiss National Science
Foundation starting grant
``Multiresolution methods for unstructured data'' (TMSGI2\_211684).
\bibliographystyle{plain}
\bibliography{literature}

@article{GIvA17,
  author  = {Geronimo, J.~S. and Iliev, P. and Van Assche, W.},
  title   = {Alpert multiwavelets and {L}egendre-{A}ngelesco multiple orthogonal polynomials},
  journal = {SIAM J. Math. Anal.},
  year    = {2017},
  volume  = {49},
  number  = {1},
  pages   = {626--645},
}

@article{DDLY00,
  author  = {D.~L. Donoho and N. Dyn and D. Levin and T.~P.-Y. Yu},
  title   = {Smooth Multiwavelet Duals of {A}lpert Bases by 
             Moment-Interpolating Refinement},
  journal = {Appl. Comput. Harmon. Anal.},
  volume  = {9},
  number  = {2},
  pages   = {166--203},
  year    = {2000},
}

@article{CCS14,
  author  = {Chkifa, A. and Cohen, A. and Schwab, C.},
  title   = {High-Dimensional Adaptive Sparse Polynomial Interpolation and
             Applications to Parametric {PDE}s},
  journal = {Found. Comput. Math.},
  year    = {2014},
  volume  = {14},
  number  = {4},
  pages   = {601--633},
  doi     = {10.1007/s10208-013-9154-z}
}

@article{GG03,
  author  = {T. Gerstner and M. Griebel},
  title   = {Dimension-Adaptive Tensor-Product Quadrature},
  journal = {Computing},
  year    = {2003},
  volume  = {71},
  number  = {1},
  pages   = {65--87},
}

@book{D02,
	author = {Dudley, R.~M.},
	title = {Real Analysis and Probability},
	publisher = {Cambridge University Press},
	address = {Cambridge},
	year = {2002}
}

@article{HM22,
	author = {Harbrecht, H. and Multerer, M.},
	title = {Samplets: Construction and Scattered Data Compression},
	journal = {J. Comput. Phys.},
	year = {2022},
	volume = {471},
	pages = {111616},
}

@article{A93,
	author = {Alpert, B.~K.},
	title = {A Class of Bases in {$L^2$} for the Sparse Representation of Integral Operators},
	journal = {SIAM J. Math. Anal.},
	year = {1993},
	volume = {24},
	number = {1},
	pages = {246--262}
}

@article{ABCR93,
	author = {Alpert, B.~K. and Beylkin, G. and Coifman, R.~R. and Rokhlin, V.},
	title = {Wavelet-Like Bases for the Fast Solution of Second-Kind Integral Equations},
	journal = {SIAM J. Sci. Comput.},
	year = {1993},
	volume = {14},
	number = {1},
	pages = {159--184}
}

@book{D92,
	author = {Daubechies, I.},
	title = {Ten Lectures on Wavelets},
	publisher = {SIAM},
	address = {Philadelphia, PA},
	year = {1992},
}

@book{M08,
	author = {Mallat, S.},
	title = {A Wavelet Tour of Signal Processing: The Sparse Way},
	edition = {3rd},
	publisher = {Academic Press},
	address = {Burlington, MA},
	year = {2008},
}

@book{M93,
	author = {Meyer, Y.},
	title = {Wavelets and Operators},
	series = {Camb. Stud. Adv. Math.},
	volume = {37},
	publisher = {Cambridge University Press},
	address = {Cambridge},
	year = {1993},
}

@article{D97,
	author = {Dahmen, W.},
	title = {Multiscale and Wavelet Methods for Operator Equations},
	journal = {Acta Numer.},
	year = {1997},
	volume = {6},
	pages = {55--228}
}

@article{EGMQ25,
	author = {Elefante, G. and Giacchi, G. and Multerer, M. and Quizi, J.},
	title = {Bespoke multiresolution analysis of graph signals},
	journal = {arXiv},
	year = {2025},
	note = {Preprint arXiv:2507.19181}
}

@article{BCR91,
	author = {Beylkin, G. and Coifman, R. and Rokhlin, V.},
	title = {Fast wavelet transforms and numerical algorithms {I}},
	journal = {Comm. Pure Appl. Math.},
	volume = {44},
	number = {2},
	pages = {141-183},
	year = {1991}
}

@ARTICLE{RV91,
	author = {Rioul, O. and Vetterli, M.},
	journal = {IEEE Signal Process. Mag.},
	title = {Wavelets and signal processing},
	year = {1991},
	volume = {8},
	number = {4},
	pages = {14--38},
}

@article{B92,
	author = {Beylkin, G.},
	title = {On the Representation of Operators in Bases of Compactly Supported Wavelets},
	journal = {SIAM J. Numer. Anal.},
	volume = {29},
	number = {6},
	pages = {1716--1740},
	year = {1992}
}

@article{DV98,
	author = {DeVore, R.~A.},
	title = {Nonlinear approximation},
	journal = {Acta Numer.},
	year = {1998},
	volume = {7},
	pages = {51--150},
}

@article{V58,
	author = {V.~S. Varadarajan},
     TITLE = {On the convergence of sample probability distributions},
   JOURNAL = {Sankhya},
  FJOURNAL = {Sankhya. The Indian Journal of Statistics.},
    VOLUME = {19},
      YEAR = {1958},
     PAGES = {23--26},
      ISSN = {0036-4452},
   MRCLASS = {60.00},
  MRNUMBER = {94839},
MRREVIEWER = {M.\ Lo\`eve},
}

@book {S67,
	AUTHOR = {Szeg\H{o}, G.},
	TITLE = {Orthogonal Polynomials},
	SERIES = {American Mathematical Society Colloquium Publications, Vol.
	XXIII},
	EDITION = {Fourth},
	PUBLISHER = {American Mathematical Society, Providence, RI},
	YEAR = {1975},
	PAGES = {xiii+432},
	MRCLASS = {42A52 (33A65)},
	MRNUMBER = {372517},
}

@book {C78,
	AUTHOR = {Chihara, T.~S.},
	TITLE = {An Introduction to Orthogonal Polynomials},
	SERIES = {Mathematics and its Applications, Vol. 13},
	PUBLISHER = {Gordon and Breach Science Publishers, New York-London-Paris},
	YEAR = {1978},
	PAGES = {xii+249},
	ISBN = {0-677-04150-0},
	MRCLASS = {42A52},
	MRNUMBER = {481884},
	MRREVIEWER = {A. G. Law},
}

@article{A98,
title = {Multiple orthogonal polynomials},
journal = {J. Comput. Appl. Math.},
volume = {99},
number = {1},
pages = {423-447},
year = {1998},
issn = {0377-0427},
author = {A.~I. Aptekarev},
}

@article{BO16,
	author = {Basu, K. and Owen, A.~B.},
	title = {Transformations and {H}ardy--{K}rause Variation},
	journal = {SIAM J. Numer. Anal.},
	volume = {54},
	number = {3},
	pages = {1946-1966},
	year = {2016},
	doi = {10.1137/15M1052184},
}

@incollection {C98,
    AUTHOR = {Caflisch, R.~E.},
     TITLE = {Monte {C}arlo and quasi-{M}onte {C}arlo methods},
 BOOKTITLE = {Acta Numerica, 1998},
    SERIES = {Acta Numer.},
    VOLUME = {7},
     PAGES = {1--49},
 PUBLISHER = {Cambridge Univ. Press, Cambridge},
      YEAR = {1998},
      ISBN = {0-521-64316-3},
   MRCLASS = {65C05 (76M35 76P05)},
  MRNUMBER = {1689431},
MRREVIEWER = {J.\ Spanier},
       DOI = {10.1017/S0962492900002804},
       URL = {https://doi.org/10.1017/S0962492900002804},
}

@book {W05,
    AUTHOR = {Wendland, H.},
     TITLE = {Scattered Data Approximation},
    SERIES = {Cambridge Monographs on Applied and Computational Mathematics},
    VOLUME = {17},
 PUBLISHER = {Cambridge University Press, Cambridge},
      YEAR = {2005},
     PAGES = {x+336},
      ISBN = {978-0521-84335-5; 0-521-84335-9},
   MRCLASS = {41-02 (41A10 41A63 65D10)},
  MRNUMBER = {2131724},
}

@article {M20,
    AUTHOR = {Mityagin, B.~S.},
     TITLE = {The zero set of a real analytic function},
   JOURNAL = {Mat. Zametki},
  FJOURNAL = {Matematicheskie Zametki},
    VOLUME = {107},
      YEAR = {2020},
    NUMBER = {3},
     PAGES = {473--475},
      ISSN = {0025-567X,2305-2880},
   MRCLASS = {26E05 (26B10)},
  MRNUMBER = {4070868},
MRREVIEWER = {Anna\ Valette-Stasica},
       DOI = {10.4213/mzm12620},
       URL = {https://doi.org/10.4213/mzm12620},
}

@book {K20,
    AUTHOR = {Klenke, A.},
     TITLE = {Probability {T}heory - A {C}omprehensive {C}ourse},
    SERIES = {Universitext},
      NOTE = {Third edition},
 PUBLISHER = {Springer},
 ADDRESS={Cham},
      YEAR = {2020},
     PAGES = {xiv+716},
}

@article{balazs2024construction,
  title={Construction of generalized samplets in {B}anach spaces},
  author={Balazs, P. and Multerer, M.},
  journal = {arXiv},
  note = {Preprint arXiv:2412.00954},
  year={2025}
}

@article {DL04,
    AUTHOR = {Dekel, S. and Leviatan, D.},
     TITLE = {The {B}ramble-{H}ilbert lemma for convex domains},
   JOURNAL = {SIAM J. Math. Anal.},
  FJOURNAL = {SIAM Journal on Mathematical Analysis},
    VOLUME = {35},
      YEAR = {2004},
    NUMBER = {5},
     PAGES = {1203--1212},
      ISSN = {0036-1410,1095-7154},
  MRNUMBER = {2050198},
       DOI = {10.1137/S0036141002417589},
       URL = {https://doi.org/10.1137/S0036141002417589},
}

@book {C03,
    AUTHOR = {Cohen, A.},
     TITLE = {Numerical {A}nalysis of {W}avelet {M}ethods},
    SERIES = {Studies in Mathematics and its Applications},
    VOLUME = {32},
 PUBLISHER = {North-Holland Publishing Co., Amsterdam},
      YEAR = {2003},
     PAGES = {xviii+336},
      ISBN = {0-444-51124-5},
   MRCLASS = {65T60 (42C40 94A08)},
  MRNUMBER = {1990555},
}

@article {BH70,
    AUTHOR = {Bramble, J. H. and Hilbert, S. R.},
     TITLE = {Estimation of linear functionals on {S}obolev spaces with
              application to {F}ourier transforms and spline interpolation},
   JOURNAL = {SIAM J. Numer. Anal.},
  FJOURNAL = {SIAM Journal on Numerical Analysis},
    VOLUME = {7},
      YEAR = {1970},
     PAGES = {112--124},
      ISSN = {0036-1429},
   MRCLASS = {65.20 (41.00)},
  MRNUMBER = {263214},
MRREVIEWER = {J.\ W.\ Jerome},
       DOI = {10.1137/0707006},
       URL = {https://doi.org/10.1137/0707006},
}
\appendix

\section{Rates, discrepancy, and deterministic bounds}
\label{app:rates_discrepancy}

\subsection{Star discrepancy and the Koksma-Hlawka inequality}
Let \(D=[0,1]^d\) and \(X_N=\{\bs x_1,\ldots,\bs x_N\}\subset D\). 
The star discrepancy of \(X_N\) is defined as
\[
\mathrm{disc}_N^*(X_N)\isdef
\sup_{\bs t\in[0,1]^d}
\left|
\frac{1}{N}\sum_{n=1}^N\mathbbm 1_{[\bs 0,\bs t)}(\bs x_n)-\prod_{i=1}^d t_i
\right|,
\quad
[\bs 0,\bs t)\isdef \prod_{i=1}^d[0,t_i).
\]
If \(f\colon D\to\Rbb\) has finite Hardy-Krause variation \(V_{\mathrm{HK}}(f)\), 
then the Koksma-Hlawka inequality, see, for example, \cite{C98}, 
yields
\begin{equation}\label{eq:KH_app}
\left|
\frac{1}{N}\sum_{n=1}^N f(\bs x_n)-\int_D f(\bs x)\d\bs x
\right|
\le
\mathrm{disc}_N^*(X_N)V_{\mathrm{HK}}(f).
\end{equation}

\subsection{Hardy-Krause variation of polynomials}
\label{app:BVHK_polys}
We report on a sufficient condition for bounded Hardy-Krause variation and 
apply it to polynomials.

For a nonempty subset \(u\subseteq\{1,\ldots,d\}\), let
\(\partial_u f\) denote the mixed first-order partial derivative obtained 
by differentiating once with respect to every variable indexed by \(u\). 
For \(\bs x_u\in[0,1]^{|u|}\), define \(\bs x_{u}:\bs 1_{-u}\in[0,1]^d\) 
by setting
\[
(\bs x_{u}:\bs 1_{-u})_j=x_j \quad\text{if }j\in u,
\quad
(\bs x_{u}:\bs 1_{-u})_j=1 \quad\text{if }j\notin u.
\]

The details concerning the following lemmas can be found in~\cite{BO16}.

\begin{lemma}\label{lem:smooth_BVHK}
Let \(f\colon D\to\Rbb\) be such that \(\partial_u f\) exists and is 
Lebesgue integrable on \(D\) for every nonempty subset
\(u\subseteq\{1,\ldots,d\}\). Then, \(f\) has finite Hardy-Krause 
variation and satisfying
\begin{equation}\label{eq:BO16_HK_bound}
V_{\mathrm{HK}}(f)
\le
\sum_{\emptyset\neq u\subseteq\{1,\ldots,d\}}
\int_{[0,1]^{|u|}}
\bigl|\partial_u f(\bs x_{u}:\bs 1_{-u})\bigr|\d\bs x_u.
\end{equation}
\end{lemma}

\begin{lemma}\label{lem:poly_BVHK}
Let \(p\colon D\to\Rbb\) be a polynomial on \(D\). Then \(V_{\mathrm{HK}}(p)<\infty\).
Moreover, for any two such polynomials \(p,q\), we have \(V_{\mathrm{HK}}(pq)<\infty\).
\end{lemma}

\begin{remark}\label{rem:BVHK_cuboid}
If \(D=\prod_{i=1}^d[a_i,b_i]\) is a cuboid, then by an affine change of variables
\[
\bs x=\bs a+\mathrm{diag}(\bs b-\bs a)\bs y,
\quad \bs y\in[0,1]^d,
\]
one reduces boundedness of \(V_{\mathrm{HK}}\) on \(D\) to boundedness 
on \([0,1]^d\). In particular, every polynomial restricted to a cuboid 
has finite Hardy-Krause variation.
\end{remark}

\subsection{Coefficient and uniform bounds for empirical 
monic orthogonal polynomials}
\label{app:coeff_rates}

In this appendix we use the Hardy-Krause variation on $[0,1]^d$,
as in Lemma~\ref{lem:smooth_BVHK}, together with the
Koksma-Hlawka inequality~\eqref{eq:KH_app}, whose integral term is with
respect to Lebesgue measure. Consequently, Lemma~\ref{lem:coeff_rate_multivariate_app}
and Corollary~\ref{cor:rate_orthonormal_app_app} are stated under the implicit
assumption that $\Pbb$ is the Lebesgue measure on $D=[0,1]^d$.
If one wishes to treat a general probability measure $\Pbb$, then one has to resort
to the discrepancy and the Hardy-Krause variation defined with respect to
$\Pbb$. This is a different notion, and finiteness for polynomials does not
automatically follow from Lemma~\ref{lem:poly_BVHK}.

\begin{lemma}\label{lem:coeff_rate_multivariate_app}
Let \(\{\pi_{\bs\alpha_i}\}_{i=1}^{|\Lambda_k|}\) and 
\(\{\pi_{\bs\alpha_i}^{(N)}\}_{i=1}^{|\Lambda_k|}\) be the monic orthogonal 
polynomials obtained by orthogonalizing the ordered monomials 
\(\{\bs x^{\bs\alpha_i}\}_{i=1}^{|\Lambda_k|}\) with respect to 
the \((\cdot,\cdot)_\Pbb\)-inner product and 
the \((\cdot,\cdot)_{\widehat{\Pbb}_N}\)-inner product, 
respectively, i.e.,
\[
\pi_{\bs\alpha_i}(\bs x)
=
\bs x^{\bs\alpha_i}
+\sum_{j=1}^{i-1} \ell_{i,j}\bs x^{\bs\alpha_j},
\quad
(\pi_{\bs\alpha_i},\bs x^{\bs\alpha_r})_\Pbb=0,
\quad r=1,\ldots,i-1,
\]
and similarly
\[
\pi_{\bs\alpha_i}^{(N)}(\bs x)
=
\bs x^{\bs\alpha_i}
+\sum_{j=1}^{i-1} \ell_{i,j}^{(N)}\bs x^{\bs\alpha_j},
\quad
(\pi_{\bs\alpha_i}^{(N)},\bs x^{\bs\alpha_r})_{\widehat{\Pbb}_N}=0,
\quad r=1,\ldots,i-1.
\]

Assume that the set \(X_N\subset[0,1]^d\) satisfies
\begin{equation}\label{eq:disc_to_zero_assump}
\mathrm{disc}_N^*(X_N)\xrightarrow{N\to\infty} 0.
\end{equation}
Then, for every \(i\in\{1,\ldots,|\Lambda_k|\}\), 
there exist constants \(C_i>0\), \(C_i'>0\) independent of \(N\), 
and \(N_i\in\Nbb_0\) such that for all \(N\ge N_i\),
\begin{equation}\label{eq:coeff_rate_vec_app}
\left(
\sum_{j=1}^{i-1}|\ell_{i,j}^{(N)}-\ell_{i,j}|^2
\right)^{1/2}
\le
C_i\mathrm{disc}_N^*(X_N),
\end{equation}
and consequently
\begin{equation}\label{eq:sup_rate_monic_app_app}
\|\pi_{\bs\alpha_i}^{(N)}-\pi_{\bs\alpha_i}\|_\infty
\le
C_i'\mathrm{disc}_N^*(X_N).
\end{equation}
\end{lemma}

\begin{proof}
Fix \(i\) and define
\(
\mathcal V_i\isdef \spn\{\bs x^{\bs\alpha_1},\ldots,\bs x^{\bs\alpha_{i-1}}\}.
\)
Write
\[
\pi_{\bs\alpha_i}=\bs x^{\bs\alpha_i}+q_i,
\quad
\pi_{\bs\alpha_i}^{(N)}=\bs x^{\bs\alpha_i}+q_i^{(N)},
\]
with \(q_i,q_i^{(N)}\in\mathcal V_i\), and set
\(
w\isdef q_i^{(N)}-q_i\in\mathcal V_i.
\)

The orthogonality relations for \(\pi_{\bs\alpha_i}\) and 
\(\pi_{\bs\alpha_i}^{(N)}\) can equivalently be written as
\begin{equation}\label{eq:var_mu_app_app}
(\bs x^{\bs\alpha_i}+q_i,v)_\Pbb=0
\quad\text{for all } v\in\mathcal V_i,
\end{equation}
and
\begin{equation}\label{eq:var_muN_app_app}
(\bs x^{\bs\alpha_i}+q_i^{(N)},v)_{\widehat{\Pbb}_N}=0
\quad \text{for all } v\in\mathcal V_i.
\end{equation}
Since \(q_i^{(N)}=q_i+w\), the second identity becomes
\[
(\bs x^{\bs\alpha_i}+q_i+w,v)_{\widehat{\Pbb}_N}=0
\quad \text{for all } v\in\mathcal V_i,
\]
hence
\begin{equation}\label{eq:w_identity_step1}
(w,v)_{\widehat{\Pbb}_N}
=
-(\bs x^{\bs\alpha_i}+q_i,v)_{\widehat{\Pbb}_N}
\quad\text{for all } v\in\mathcal V_i.
\end{equation}
Using also~\eqref{eq:var_mu_app_app}, we obtain
\[
(w,v)_{\widehat{\Pbb}_N}
=
(\bs x^{\bs\alpha_i}+q_i,v)_\Pbb
-
(\bs x^{\bs\alpha_i}+q_i,v)_{\widehat{\Pbb}_N}
\quad \text{for all } v\in\mathcal V_i.
\]
Choosing \(v=w\) yields the identity
\begin{equation}\label{eq:energy_identity_app_app}
\|w\|_{\widehat{\Pbb}_N}^2
=
(\bs x^{\bs\alpha_i}+q_i,w)_\Pbb -
(\bs x^{\bs\alpha_i}+q_i,w)_{\widehat{\Pbb}_N}.
\end{equation}

We now estimate the left-hand side from below. Write
\[
w=\sum_{j=1}^{i-1} c_j\bs x^{\bs\alpha_j},
\quad
\bs c=[c_j]_{j=1}^{i-1}.
\]
Since the Gram matrix \(\bs G_{i-1}
      =[(\bs x^{{\bs\alpha}_j},\bs x^{{\bs\alpha}_r})_\Pbb]_{j,r=1}^{i-1}\)
      is positive definite, 
see Remark~\ref{Rem:positive_emp_gr}, there exists \(\lambda_i>0\), 
depending only on \(\bs G_{i-1}\), such that
\begin{equation}\label{eq:coercivity_app_app}
\|w\|_\Pbb^2=\bs c^\intercal \bs G_{i-1} \bs c
\ge\lambda_i\|\bs c\|_2^2.
\end{equation}

Next apply the Koksma-Hlawka inequality~\eqref{eq:KH_app} to \(w^2\). This gives
\[
\big|\|w\|_{\widehat{\Pbb}_N}^2-\|w\|_\Pbb^2\big|
=
\bigg|\frac{1}{N}\sum_{n=1}^N w^2(\bs x_n)-\int_D w^2(\bs x)\d\bs x\bigg|
\le
\mathrm{disc}_N^*(X_N) V_{\mathrm{HK}}(w^2).
\]
By Lemma~\ref{lem:poly_BVHK}, the polynomial \(w^2\) is of 
bounded Hardy-Krause variation. Moreover,
\[
w^2=\sum_{r,s=1}^{i-1} c_r c_s \bs x^{\bs\alpha_r+\bs\alpha_s},
\]
and therefore
\[
V_{\mathrm{HK}}(w^2)
\le
\sum_{r,s=1}^{i-1}|c_r||c_s| V_{\mathrm{HK}}(\bs x^{\bs\alpha_r+\bs\alpha_s})
\le
K_i\|\bs c\|_2^2,
\]
for some constant \(K_i>0\) depending only on the finite family
\(\{V_{\mathrm{HK}}(\bs x^{\bs\alpha_r+\bs\alpha_s})\}_{r,s\le i-1}\).
Combining this with~\eqref{eq:coercivity_app_app}, we obtain
\[
\|w\|_{\widehat{\Pbb}_N}^2
\ge
\bigl(\lambda_i-K_i\mathrm{disc}_N^*(X_N)\bigr)\|\bs c\|_2^2.
\]
By assumption~\eqref{eq:disc_to_zero_assump}, we may choose \(N_i\) such that
\[
\mathrm{disc}_N^*(X_N)\le \frac{\lambda_i}{2K_i}
\quad\text{for all }N\ge N_i.
\]
Hence, for all \(N\ge N_i\),
\begin{equation}\label{eq:coercive_muN_app_app}
\|w\|_{\widehat{\Pbb}_N}^2
\ge\frac{\lambda_i}{2}\|\bs c\|_2^2.
\end{equation}

We next estimate the right-hand side of~\eqref{eq:energy_identity_app_app}. 
Write \(
q_i=\sum_{j=1}^{i-1} \ell_{i,j}\bs x^{\bs\alpha_j}.
\)
Then \((\bs x^{\bs\alpha_i}+q_i)w\) is a linear combination of monomials of 
the form \(\bs x^{\bs\alpha_i+\bs\alpha_r}\) and 
\(\bs x^{\bs\alpha_j+\bs\alpha_r}\), with \(r,j\le i-1\). 
By Lemma~\ref{lem:poly_BVHK}, all these monomials are bounded 
in Hardy-Krause sense, so
\[
V_{\mathrm{HK}}\bigl((\bs x^{\bs\alpha_i}+q_i)w\bigr)
\le
\sum_{r=1}^{i-1}|c_r| V_{\mathrm{HK}}(\bs x^{\bs\alpha_i+\bs\alpha_r})
+
\sum_{j=1}^{i-1}\sum_{r=1}^{i-1}|\ell_{i,j}||c_r|
V_{\mathrm{HK}}(\bs x^{\bs\alpha_j+\bs\alpha_r})
\le
B_i\|\bs c\|_2,
\]
for some constant \(B_i>0\) depending on \(i\), 
the finite family of Hardy-Krause variations, and the fixed 
coefficients \(\ell_{i,j}\).
	
Applying the Koksma-Hlawka inequality to 
\((\bs x^{\bs\alpha_i}+q_i)w\) and using~\eqref{eq:energy_identity_app_app}, 
we obtain
\begin{equation}\label{eq:upper_energy_app_app}
\|w\|_{\widehat{\Pbb}_N}^2
\le
\mathrm{disc}_N^*(X_N)B_i\|\bs c\|_2.
\end{equation}
	
Finally, for \(N\ge N_i\), combining~\eqref{eq:coercive_muN_app_app} 
and~\eqref{eq:upper_energy_app_app} gives
\[
\frac{\lambda_i}{2}\|\bs c\|_2^2
\le
\mathrm{disc}_N^*(X_N)B_i\|\bs c\|_2.
\]	
If \(\|\bs c\|_2=0\), there is nothing to prove. Otherwise, 
dividing by \(\|\bs c\|_2\) yields
\[
\|\bs c\|_2
\le
\frac{2B_i}{\lambda_i}\mathrm{disc}_N^*(X_N).
\]
Since \(c_j=\ell_{i,j}^{(N)}-\ell_{i,j}\), this 
proves~\eqref{eq:coeff_rate_vec_app} with
\[
C_i=\frac{2B_i}{\lambda_i}.
\]
	
To conclude, note that \(|\bs x^{\bs\alpha_j}|\le 1\) on \(D=[0,1]^d\). 
Hence, there holds
\[
\|\pi_{\bs\alpha_i}^{(N)}-\pi_{\bs\alpha_i}\|_\infty
=
\|w\|_\infty
\le
\sum_{j=1}^{i-1}|\ell_{i,j}^{(N)}-\ell_{i,j}|\le
\sqrt{i-1}\|\bs c\|_2.
\]
Together with the previous bound, 
this yields~\eqref{eq:sup_rate_monic_app_app} with
\[
C_i'=\sqrt{i-1}C_i.
\]
\end{proof}

\begin{corollary}\label{cor:rate_orthonormal_app_app}
Let \(\{\pi_{\bs\alpha_i}\}_{i=1}^{|\Lambda_k|}\) and 
\(\{\pi_{\bs\alpha_i}^{(N)}\}_{i=1}^{|\Lambda_k|}\) be the monic orthogonal 
polynomials obtained by orthogonalizing the ordered monomials 
\(\{\bs x^{\bs\alpha_i}\}_{i=1}^{|\Lambda_k|}\) with respect to 
\((\cdot,\cdot)_\Pbb\) and \((\cdot,\cdot)_{\widehat{\Pbb}_N}\). Further, let
\[
\widehat\pi_{\bs\alpha_i}\isdef \frac{\pi_{\bs\alpha_i}}{\|\pi_{\bs\alpha_i}\|_\Pbb},
\quad
\widehat\pi_{\bs\alpha_i}^{(N)}\isdef
\frac{\pi_{\bs\alpha_i}^{(N)}}{\|\pi_{\bs\alpha_i}^{(N)}\|_{\widehat{\Pbb}_N}}.
\]
Then there exist constants \(\widehat C_i>0\) and \(\widehat N_i\in\Nbb_0\) 
such that for all \(N\ge \widehat N_i\),
\[
\|\widehat\pi_{\bs\alpha_i}^{(N)}-\widehat\pi_{\bs\alpha_i}\|_\infty
\le
\widehat C_i\mathrm{disc}_N^*(X_N).
\]
\end{corollary}

 \begin{proof}
 	We start from the identity
 	\begin{equation}\label{eq:split_orthonorm_clean}
 		\widehat\pi_{\bs\alpha_i}^{(N)}-\widehat\pi_{\bs\alpha_i}
 		=
 		\frac{\pi_{\bs\alpha_i}^{(N)}-\pi_{\bs\alpha_i}}
 		{\|\pi_{\bs\alpha_i}^{(N)}\|_{\widehat{\Pbb}_N}}
 		+
 		\pi_{\bs\alpha_i}
 		\left(
 		\frac1{\|\pi_{\bs\alpha_i}^{(N)}\|_{\widehat{\Pbb}_N}}
 		-
 		\frac1{\|\pi_{\bs\alpha_i}\|_\Pbb}
 		\right).
 	\end{equation}
 	Thus it remains to estimate the difference of the monic polynomials 
  and the difference of the reciprocal normalization factors.
 	
 	By Lemma~\ref{lem:coeff_rate_multivariate_app}, there exist \(C_i>0\) 
  and \(N_i\in\Nbb_0\) such that
 	\begin{equation}\label{eq:monic_sup_rate_clean}
 		\|\pi_{\bs\alpha_i}^{(N)}-\pi_{\bs\alpha_i}\|_\infty
 		\le
 		C_i\mathrm{disc}_N^*(X_N),
 		\quad N\ge N_i.
 	\end{equation}
 	
 	Next, we compare the squared norms
 	\begin{align}
 		\bigl|
 		\|\pi_{\bs\alpha_i}^{(N)}\|_{\widehat{\Pbb}_N}^2
 		-
 		\|\pi_{\bs\alpha_i}\|_\Pbb^2
 		\bigr|
 		&=
 		\left|
 		\int_D (\pi_{\bs\alpha_i}^{(N)})^2\d\widehat{\Pbb}_N
 		-
 		\int_D \pi_{\bs\alpha_i}^2\d\Pbb
 		\right| \nonumber\\
 		&\le
 		\left|
 		\int_D\bigl((\pi_{\bs\alpha_i}^{(N)})^2-\pi_{\bs\alpha_i}^2\bigr)
    \d\widehat{\Pbb}_N
 		\right|
 		+
 		\left|
 		\int_D \pi_{\bs\alpha_i}^2\d\widehat{\Pbb}_N
 		-
 		\int_D \pi_{\bs\alpha_i}^2\d\Pbb
 		\right|.
 		\label{eq:norm_split_explicit}
 	\end{align}
 	Since \(\widehat{\Pbb}_N\) is a probability measure,
 	\[
 	\left|
 	\int_D\bigl((\pi_{\bs\alpha_i}^{(N)})^2-\pi_{\bs\alpha_i}^2\bigr)
  \d\widehat{\Pbb}_N
 	\right|
 	\le
 	\|(\pi_{\bs\alpha_i}^{(N)})^2-\pi_{\bs\alpha_i}^2\|_\infty
 	\le
 	\bigl(\|\pi_{\bs\alpha_i}^{(N)}\|_\infty+\|\pi_{\bs\alpha_i}\|_\infty\bigr)
 	\|\pi_{\bs\alpha_i}^{(N)}-\pi_{\bs\alpha_i}\|_\infty.
 	\]
 	Moreover, by the triangle inequality, we have
 	\[
 	\|\pi_{\bs\alpha_i}^{(N)}\|_\infty
 	\le
 	\|\pi_{\bs\alpha_i}\|_\infty
 	+
 	\|\pi_{\bs\alpha_i}^{(N)}-\pi_{\bs\alpha_i}\|_\infty.
 	\]
 	Hence, by~\eqref{eq:disc_to_zero_assump} 
  and~\eqref{eq:monic_sup_rate_clean}, after possibly increasing \(N_i\), 
  we may assume that
 	\[
 	\|\pi_{\bs\alpha_i}^{(N)}\|_\infty \le 2\|\pi_{\bs\alpha_i}\|_\infty,
 	\quad N\ge N_i.
 	\]
 	Therefore
 	\begin{equation}\label{eq:first_norm_term_bound}
 		\left|
 		\int_D\bigl((\pi_{\bs\alpha_i}^{(N)})^2-\pi_{\bs\alpha_i}^2\bigr)
    \d\widehat{\Pbb}_N
 		\right|
 		\le
 		3\|\pi_{\bs\alpha_i}\|_\infty C_i\mathrm{disc}_N^*(X_N),
 		\quad N\ge N_i.
 	\end{equation}
 	
 	For the second term of~\eqref{eq:norm_split_explicit}, we apply the 
  Koksma-Hlawka inequality~\eqref{eq:KH_app} to the polynomial 
  \(\pi_{\bs\alpha_i}^2\), which gives
 	\begin{equation}\label{eq:second_norm_term_bound}
 		\left|
 		\int_D \pi_{\bs\alpha_i}^2\d\widehat{\Pbb}_N
 		-
 		\int_D \pi_{\bs\alpha_i}^2\d\Pbb
 		\right|
 		\le
 		\mathrm{disc}_N^*(X_N)V_{\mathrm{HK}}(\pi_{\bs\alpha_i}^2).
 	\end{equation}
 	Since \(\pi_{\bs\alpha_i}^2\) is a polynomial on \([0,1]^d\), 
  we have \(V_{\mathrm{HK}}(\pi_{\bs\alpha_i}^2)<\infty\). 
  Combining~\eqref{eq:norm_split_explicit},~\eqref{eq:first_norm_term_bound}
  and~\eqref{eq:second_norm_term_bound}, we obtain
 	\begin{equation}\label{eq:normsq_rate_explicit}
 		\bigl|
 		\|\pi_{\bs\alpha_i}^{(N)}\|_{\widehat{\Pbb}_N}^2
 		-
 		\|\pi_{\bs\alpha_i}\|_\Pbb^2
 		\bigr|
 		\le
 		\widehat{C}\mathrm{disc}_N^*(X_N),
 		\quad N\ge N_i,
 	\end{equation}
 	where
 	\[
 	\widehat{C}
 	\isdef
 	3\|\pi_{\bs\alpha_i}\|_\infty C_i + V_{\mathrm{HK}}(\pi_{\bs\alpha_i}^2).
 	\]
 	
 	Since \(\|\pi_{\bs\alpha_i}\|_\Pbb>0\) and \(\mathrm{disc}_N^*(X_N)\to0\), 
  there exists \(N'\ge N_i\) such that
 	\[
 	\widehat{C}\mathrm{disc}_N^*(X_N)\le \frac12 \|\pi_{\bs\alpha_i}\|_\Pbb^2,
 	\quad N\ge N'.
 	\]
 	Then~\eqref{eq:normsq_rate_explicit} implies
 	\[
 	\|\pi_{\bs\alpha_i}^{(N)}\|_{\widehat{\Pbb}_N}^2
 	\ge
 	\|\pi_{\bs\alpha_i}\|_\Pbb^2
 	-
 	\widehat{C}\mathrm{disc}_N^*(X_N)
 	\ge
 	\frac12 \|\pi_{\bs\alpha_i}\|_\Pbb^2,
 	\quad N\ge N',
 	\]
 	and hence
 	\begin{equation}\label{eq:denom_lower_explicit}
 		\|\pi_{\bs\alpha_i}^{(N)}\|_{\widehat{\Pbb}_N}
 		\ge
 		\frac{1}{\sqrt2}\|\pi_{\bs\alpha_i}\|_\Pbb,
 		\quad N\ge N'.
 	\end{equation}
 	
 	It remains to control the difference of the reciprocal normalization factors. 
  We start from the identity
 	\[
 	\bigl|
 	\|\pi_{\bs\alpha_i}^{(N)}\|_{\widehat{\Pbb}_N}
 	-
 	\|\pi_{\bs\alpha_i}\|_\Pbb
 	\bigr|
 	=
 	\frac{
 		\bigl|
 		\|\pi_{\bs\alpha_i}^{(N)}\|_{\widehat{\Pbb}_N}^2
 		-
 		\|\pi_{\bs\alpha_i}\|_\Pbb^2
 		\bigr|
 	}{
 		\|\pi_{\bs\alpha_i}^{(N)}\|_{\widehat{\Pbb}_N}
 		+
 		\|\pi_{\bs\alpha_i}\|_\Pbb
 	}.
 	\]
 	Hence, by~\eqref{eq:normsq_rate_explicit} 
  and~\eqref{eq:denom_lower_explicit}, there exists a constant \(\widehat{C}'>0\) 
  such that
 	\begin{equation}\label{eq:norm_difference_rate_explicit}
 		\bigl|
 		\|\pi_{\bs\alpha_i}^{(N)}\|_{\widehat{\Pbb}_N}
 		-
 		\|\pi_{\bs\alpha_i}\|_\Pbb
 		\bigr|
 		\le
 		\widehat{C}'\mathrm{disc}_N^*(X_N),
 		\quad N\ge N'.
 	\end{equation}
 	Therefore
 	\begin{equation}\label{eq:recip_rate_explicit}
 		\left|
 		\frac1{\|\pi_{\bs\alpha_i}^{(N)}\|_{\widehat{\Pbb}_N}}
 		-
 		\frac1{\|\pi_{\bs\alpha_i}\|_\Pbb}
 		\right|
 		=
 		\frac{
 			\bigl|
 			\|\pi_{\bs\alpha_i}^{(N)}\|_{\widehat{\Pbb}_N}
 			-
 			\|\pi_{\bs\alpha_i}\|_\Pbb
 			\bigr|
 		}{
 			\|\pi_{\bs\alpha_i}^{(N)}\|_{\widehat{\Pbb}_N}\|\pi_{\bs\alpha_i}\|_\Pbb
 		}
 		\le
 		\widehat{C}''\mathrm{disc}_N^*(X_N),
 		\quad N\ge N',
 	\end{equation}
 	for a suitable constant \(\widehat{C}''>0\).
 	
 	Finally, taking \(L^\infty\)-norms in~\eqref{eq:split_orthonorm_clean} 
  and using~\eqref{eq:monic_sup_rate_clean},~\eqref{eq:denom_lower_explicit}
  and~\eqref{eq:recip_rate_explicit}, we obtain for \(N\ge N'\),
 	\[
 	\|\widehat\pi_{\bs\alpha_i}^{(N)}-\widehat\pi_{\bs\alpha_i}\|_\infty
 	\le
 	\frac{\|\pi_{\bs\alpha_i}^{(N)}-\pi_{\bs\alpha_i}\|_\infty}
 	{\|\pi_{\bs\alpha_i}^{(N)}\|_{\widehat{\Pbb}_N}}
 	+
 	\|\pi_{\bs\alpha_i}\|_\infty
 	\left|
 	\frac1{\|\pi_{\bs\alpha_i}^{(N)}\|_{\widehat{\Pbb}_N}}
 	-
 	\frac1{\|\pi_{\bs\alpha_i}\|_\Pbb}
 	\right|
 	\le
 	\widehat C_i\mathrm{disc}_N^*(X_N),
 	\]
 	for a suitable constant \(\widehat C_i>0\). Setting \(\widehat N_i\isdef N'\) 
  completes the proof.
 \end{proof}

\end{document}